\documentclass[12pt,reqno]{amsart}
\usepackage[margin=1in]{geometry}
\usepackage{amsmath,amssymb,amsthm,graphicx,amsxtra, setspace}
\usepackage[utf8]{inputenc}
\usepackage{mathrsfs}
\usepackage{hyperref}
\usepackage{upgreek}
\usepackage{mathtools}
\allowdisplaybreaks
\newtheorem{thm}{Theorem}[section]
\newtheorem{lem}{Lemma}[section]
\newtheorem{rem}{Remark}[section]

\newtheorem{Def}{Definition}[section]

\newtheorem{Ass}{Assumption}[section]
\usepackage{latexsym}
\usepackage{hyperref}
\usepackage{graphicx}
\usepackage{epstopdf}

\allowdisplaybreaks

\let\originalleft\left
\let\originalright\right
\renewcommand{\left}{\mathopen{}\mathclose\bgroup\originalleft}
\renewcommand{\right}{\aftergroup\egroup\originalright}

\newcommand{\Addresses}{{
		\footnote{

				\noindent	 \textsuperscript{1,3}Department of Mathematics, PSG College of Arts and Science,  Coimbatore, 641 046, India.

			\noindent  \textit{e-mail\textsuperscript{1}:} \texttt{ravikumarkpsg@gmail.com.}
			
			\noindent  \textit{e-mail\textsuperscript{3}:} \texttt{angurajpsg@yahoo.com.}

					\noindent \textsuperscript{2}Department of Mathematics, Indian Institute of Technology Roorkee-IIT Roorkee,
				Haridwar Highway, Roorkee, Uttarakhand 247667, India.\par\nopagebreak
				\noindent  \textit{e-mail\textsuperscript{3}:} \texttt{maniltmohan@ma.iitr.ac.in, maniltmohan@gmail.com.}

			\noindent \textsuperscript{*}Corresponding author.

			\textit{Key words:} Non-autonomous differential equations, Approximate controllability, Schauder's fixed point theorem, Resolvent operators.
			
			Mathematics Subject Classification (2010): 34K06, 34A12, 37L05, 93B05.

		{\bf Acknowledgments:} M. T. Mohan would  like to thank the Department of Science and Technology (DST), India for Innovation in Science Pursuit for Inspired Research (INSPIRE) Faculty Award (IFA17-MA110).

}}}

\begin{document}
\title[Approximate controllability of a non-autonomous differential equation]{Approximate controllability of a non-autonomous evolution equation in Banach spaces	\Addresses}
\author [ K. Ravikumar, M. T. Mohan and A. Anguraj]{K. Ravikumar\textsuperscript{1}, Manil T. Mohan\textsuperscript{2*}  and A. Anguraj\textsuperscript{3}}
 \maketitle{}
\begin{abstract}
In this paper, we consider a non-autonomous nonlinear evolution equation in separable, reflexive Banach spaces. First, we consider a linear problem and establish the approximate controllability results by finding a feedback control with the help of an optimal control problem. We then establish the approximate controllability results for a semilinear differential  equation in Banach spaces using the theory of linear evolution systems, properties of resolvent operator and Schauder's fixed point theorem. Finally, we provide an example of a non-autonomous, nonlinear diffusion equation in Banach spaces to validate the results we obtained.
\end{abstract}
\section{Introduction}\label{sec1}\setcounter{equation}{0}
The concept of controllability plays an important role in the analysis and design of control systems. Controllability of the deterministic and stochastic dynamical control system in infinite-dimensional spaces is well developed using different kinds of approaches, and the details can be found in various paper see for example \cite{AB,NA,M1,J}, etc and the references therein. From the mathematical point of view, in infinite dimensions, the problems of exact and approximate controllability are to be distinguished. Exact controllability enables to steer the system to arbitrary final state (see \cite{MTM}), while approximate controllability means that the system can be steered to arbitrary small neighborhood of final state. Approximate controllable systems are more prevalent and very often approximate controllability is completely adequate in applications, see for instance \cite{CR,M,AB,J,KS}, etc. Therefore, it is important, in fact necessary to study the weaker concept of controllability, namely approximate controllability for nonlinear systems.

 In the recent literature, there have been a few papers on the approximate controllability of the nonlinear evolution systems under different conditions, see for example \cite{M,RS,RS1,RS2}, etc. In \cite{DM}, Dauer and Mahmudov investigated the approximate controllability of a functional differential equation with compact semigroup, using the Schauder's fixed point theorem. The authors used the Banach fixed theorem to obtain the approximate controllability results, when the semigroup is not compact. Fu and Mei \cite{XF} examined the approximate controllability of semilinear neutral functional differential systems with finite delay. Later, Sakthivel et.al., \cite{G} studied the controllability of a semilinear integrodifferential equation in Banach spaces. Approximate controllability of non-autonomous semilinear systems in Hilbert spaces  with various conditions can be obtained from \cite{GRK,XF2}, etc. Fu in \cite{XF1} investigated the approximate controllability of semilinear non-autonomous evolution systems  in Hilbert spaces with state-dependent delay.  Using the resolvent operators, the approximate controllability results for fractional differential equations in Hilbert spaces is explored by Fan in \cite{F}. Mishra and Sharma in \cite{MI} investigated the approximate controllability of a non-autonomous functional differential equation in Hilbert spaces using the theory of linear evolution system, Schauder's fixed point theorem, and using resolvent operators. Chen et. al. in \cite{CZL} obtained the approximate controllability for a class of non-autonomous evolution system of parabolic type with nonlocal conditions in Banach spaces. But the resolvent operator defined in the work  \cite{CZL} is applicable only in the case of Hilbert spaces (see \eqref{2.2} below). So, it appears to the authors that the results announced in the paper \cite{CZL} are valid only in separable Hilbert spaces.  We make use of the techniques adopted in \cite{F1,F,MI} to establish the approximate controllability of a non-autonomous nonlinear evolution equation in reflexive, separable  Banach spaces. The novelty of the work is that, it provides a systematic approach of approximate controllability of non-autonomous nonlinear evolution systems in reflexive Banach spaces. 
 
 Let $\mathbb{X}$ be a separable, reflexive Banach space (with a strictly convex dual) and $\mathbb{H}$ be a separable Hilbert space. In this paper, we examine the approximate controllability of the following non-autonomous, nonlinear  evolution differential system:   
\begin{equation}\label{1}
\left\{
\begin{aligned}
\dot{x}(t)&= \mathrm{A}(t)x(t)+f(t,x(t))+\mathrm{B}u(t),\ t\in[0,T],\\
x(0)&=x_{0},
\end{aligned}
\right.
\end{equation} 
where $f:[0,T]\times \mathbb{X}\rightarrow \mathbb{X}$, $\mathrm{A}(\cdot)$ is a linear operator on $\mathbb{X}$, $\mathrm{B}$ is a  bounded linear operator from $\mathbb{H}$ to $\mathbb{X}$ and $x_0\in\mathbb{X}$. The control function $u(\cdot)$ is given in space $\mathrm{L}^{2}([0,T];\mathbb{H}),$ which is a Hilbert space of admissible control functions.

The rest of the paper is organized as follows. In the next section, we provide some necessary definitions and results required to develop the theory for the approximate controllability of the non-autonomous system \eqref{1}. Section \ref{sec3} is devoted for the approximate controllability of  linear  problem corresponding to the system \eqref{1}. In order to do this obtain the control in feedback form, we first formulate an optimal control problem and establish the existence of an optimal control (Theorem \ref{optimal}).  Using this optimal control, we derive the feedback control needed to establish the approximate controllability of linear non-autonomous system (Lemma \ref{lem3.1} and Theorem \ref{thm3.2}). In section \ref{sec4}, we consider the the approximate controllability of non-autonomous nonlinear  evolution differential system. We make use of the method of resolvent operators and Schauder's fixed point theorem to study the approximate controllability of a non-autonomous evolution equation in reflexive Banach spaces (Theorems \ref{thm4.1} and \ref{thm4.2}). Finally, in section \ref{sec5}, we give an example of a non-autonomous, nonlinear diffusion equation to validate the theory that we developed in sections \ref{sec3} and \ref{sec4}.

\section{Preliminaries}\label{sec2}\setcounter{equation}{0}
In this section, we introduce some basic definitions and notations, which are going to be used throughout the paper. As discussed in the previous section, $\mathbb{X}$ denotes a separable, reflexive Banach space and $\mathbb{H}$ denotes a separable Hilbert space.  The norms in $\mathbb{X}$, $\mathbb{X}'$ and $\mathbb{H}$ are denoted by $\|\cdot\|_{\mathbb{X}}$, $\|\cdot\|_{\mathbb{X}'}$ and $\|\cdot\|_{\mathbb{H}}$, respectively. The inner product in $\mathbb{H}$ is denoted by $(\cdot,\cdot)$ and the duality pairing between $\mathbb{X}$ and its topological dual $\mathbb{X}'$ is denoted by $\langle\cdot,\cdot\rangle$. Remember that a reflexive Banach space is separable if and only if its dual is separable. Thus, since $\mathbb{X}$ is a separable, reflexive Banach space, its dual $\mathbb{X}'$ is separable. The space of all bounded linear operators from $\mathbb{H}$ to $\mathbb{X}$ is denoted by $\mathcal{L}(\mathbb{H};\mathbb{X})$ and the operator norm is denoted by $\|\cdot\|_{\mathcal{L}(\mathbb{H};\mathbb{X})}$. By $\mathcal{L}(\mathbb{X}),$ we mean the set of all bounded linear operators defined on $\mathbb{X}$ and the operator norm is denoted by $\|\cdot\|_{\mathcal{L}(\mathbb{X})}$. Also, we denote $\mathcal{K}(\mathbb{X})$ as the space of all compact linear operators on $\mathbb{X}$. 

\subsection{The duality mapping} We define a mapping $\mathrm{J}:\mathbb{X}\to2^{\mathbb{X}'}$ by (see \cite{VB})
$$\mathrm{J}[x]=\left\{x'\in\mathbb{X}':\langle x,x'\rangle=\|x\|_{\mathbb{X}}^2=\|x'\|_{\mathbb{X}'}^2 \right\}, \ \text{ for all } \ x\in\mathbb{X}.$$ The mapping $\mathrm{J}$ is called the \emph{duality mapping} of the space $\mathbb{X}$. Note that duality map on $\mathbb{X}$ satisfies $\mathrm{J}[\lambda x] = \lambda \mathrm{J}[x]$, for all $\lambda\in\mathbb{R}$ and $x\in\mathbb{X}$. Moreover, $\mathrm{J}^{-1}:\mathbb{X}'\to\mathbb{X}$ is also a duality mapping. For example, if $\mathbb{X}=\mathbb{H}$ is a Hilbert space identified with its own dual, then $\mathrm{J}=\mathrm{I}$, the identity operator in $\mathbb{H}$. If $\mathbb{X}=\mathbb{L}^p(\Omega)$, where $1<p<\infty$ and $\Omega$ is a measurable subset of $\mathbb{R}^n$. Then the duality mapping of $\mathbb{X}$ is given by 
\begin{align*}
\mathrm{J}[v](y)=|v(y)|^{p-2}v(y)\|v\|_{\mathbb{L}^p(\Omega)}^{2-p}, \ \text{ a.e. }\ y\in\Omega, \ \text{ for all }\ v\in\mathbb{L}^p(\Omega). 
\end{align*} Since the space $\mathbb{X}$ is reflexive, $\mathbb{X}$ can be renormed such that $\mathbb{X}$ and $\mathbb{X}'$ becomes strictly convex (\cite{AE}). From the strict convexity of $\mathbb{X}'$, we obtain that the duality mapping $\mathrm{J}:\mathbb{X}\to\mathbb{X}'$  is single valued and demicontinuous, that is, $$x_k\to x\ \text{ in }\ \mathbb{X}\ \text{
implies }\ \mathrm{J}[x_k] \xrightharpoonup{w} \mathrm{J}[x]\ \text{ in } \ \mathbb{X}'.$$ Moreover, if the space $\mathbb{X}'$ is uniformly convex (that is, $\mathbb{X}$ is uniformly smooth), then $\mathrm{J}$ is uniformly continuous on every bounded subset of $\mathbb{X}$  (see Theorem 1.2, \cite{VB}). It should be noted that every uniformly convex space $\mathbb{X}$ is strictly convex and by using Milman theorem (see \cite{YK}), every uniformly convex Banach space is reflexive (that is, $\mathbb{X}''=\mathbb{X}$). 

Let us now discuss about the differentiability of the map $x\mapsto\frac{1}{2}\|x\|_{\mathbb{X}}^2$. Let $\phi: \mathbb{X}\to\mathbb{R}$ be defined by $\phi(x)=\frac{1}{2}\|x\|_{\mathbb{X}}^2$. If $\mathbb{X}'$ is strictly convex then $\phi$ is G\^ateaux differentiable, and if $\mathbb{X}'$ is uniformly convex, then $\phi$ is Fr\'echet differentiable. In both cases the derivative is the duality map (see Theorem 2.1, \cite{BJM}). That is, we have $$\langle\partial_x\phi(x),y\rangle=\frac{1}{2}\frac{\mathrm{d}}{\mathrm{d}\varepsilon}\|x+\varepsilon y\|_{\mathbb{X}}^2\Big|_{\varepsilon=0}=\langle\mathrm{J}[x],y\rangle,$$ for $y\in\mathbb{X}$, where $\partial_x$ denotes the G\^ateaux derivative.

\subsection{The two parameter family of semigroups} In this subsection, we construct a two parameter family of semigroup under some assumptions on the operator $\left\{\mathrm{A}(t): t\in [0,T]\right\}$. 
 These assumptions are taken from section 5.6, Chapter 5, \cite{P} (see \cite{MI} also). 
 \begin{Ass}\label{ass2.1}
 Let $\left\{\mathrm{A}(t): t\in [0,T]\right\}$ be a family of operators satisfying the following assumptions :
 \begin{enumerate}
\item [(P1)]  The linear operator $\mathrm{A}(t)$ is a closed and the domain $\mathrm{D}(\mathrm{A}(t))=\mathrm{D}$ of $\mathrm{A}(t)$ is dense in $\mathbb{X}$ and independent of $t\in [0,T]$.
\item [(P2)] For each $t\in [0,T]$, the resolvent $\mathrm{R}(\lambda,\mathrm{A}(t))$ exists for all $\lambda$ with $\text{Re }\lambda \geq 0$ and there exists $M> 0$ such that 
\begin{align*}
\left\|\mathrm{R}(\lambda,\mathrm{A}(t))\right\|_{\mathcal{L}(\mathbb{X})}\leq \frac{M}{\left|\lambda\right|+1}.
\end{align*}
\item [(P3)] There exists constants $L>0$ and $0<\alpha \leq 1$ such that for all $t,s,\tau \in [0,T]$, we have 
\begin{align*}
\left\|(\mathrm{A}(t)-\mathrm{A}(s))\mathrm{A}^{-1}(\tau)\right\|_{\mathcal{L}(\mathbb{X})}\leq L\left|t-s\right|^{\alpha}.
\end{align*}
\item [(P4)] For each $t\in [0,T]$ and some $\lambda \in \rho(\mathrm{A}(t))$, the resolvent operator $\mathrm{R}(\lambda,\mathrm{A}(t))$ is  compact.
\end{enumerate}
\end{Ass}
Let us now provide the definition of evolution system and state its properties. 
\begin{Def}[\cite{P}]
A two parameter family of bounded linear operators $$\left\{\mathrm{U}(t,s):0\leq s\leq t\leq T\right\},$$ on $\mathbb{X}$ is called an \emph{evolution system} if the following two conditions are satisfied:
\begin{enumerate}
\item [(1)] $\mathrm{U}(s,s)=\mathrm{I}$, $\mathrm{U}(t,r)\mathrm{U}(r,s)=\mathrm{U}(t,s)$ for $0\leq s\leq r \leq t \leq T$.
\item [(2)] $(t,s)\mapsto \mathrm{U}(t,s)$ is strongly continuous for $0\leq s\leq t\leq T$.
\end{enumerate}
\end{Def}
\begin{lem}[Theorem 6.1, Chapter 5, \cite{P}]\label{lem2.1}
Under the assumptions, (P1)-(P3), there is a \emph{unique evolution system} $\mathrm{U}(t,s)$ on $0\leq s\leq t\leq T$, satisfying
\begin{enumerate}
\item [(1)] For $0\leq s\leq t\leq T$, we have
\begin{align*}
\left\|\mathrm{U}(t,s)\right\|_{\mathcal{L}(\mathbb{X})}\leq C.
\end{align*}
\item [(2)] For $0\leq s\leq t\leq T$, $\mathrm{U}(t,s):\mathbb{X}\rightarrow \mathrm{D}$ and $t\mapsto \mathrm{U}(t,s)$ is strongly differentiable in $\mathbb{X}$. The derivative $\frac{\partial}{\partial t}\mathrm{U}(t,s)\in \mathcal{L}({\mathbb{X}})$ and is strongly continuous on $0\leq s\leq t\leq T$. Moreover, we also have
\begin{align*}
\frac{\partial}{\partial t}\mathrm{U}(t,s)+\mathrm{A}(t)\mathrm{U}(t,s)=0, \ \text{ for } \ 0\leq s\leq t\leq T,
\end{align*}
\begin{align*}
\left\|\frac{\partial}{\partial t}\mathrm{U}(t,s)\right\|_{\mathcal{L}(\mathbb{X})}=\left\|\mathrm{A}(t)\mathrm{U}(t,s)\right\|_{\mathcal{L}(\mathbb{X})}\leq \frac{C}{t-s},
\end{align*}
and
\begin{align*}
\left\|\mathrm{A}(t)\mathrm{U}(t,s)\mathrm{A}(s)^{-1}\right\|_{\mathcal{L}(\mathbb{X})}\leq C, \ \text{ for }\ 0\leq s\leq t\leq T.
\end{align*}
\item [(3)] For every $v\in \mathrm{D}$ and $t\in [0,T]$, $\mathrm{U}(t,s)v$ is differentiable with respect to $s$ on $0\leq s\leq t\leq T$ and
\begin{align*}
\frac{\partial}{\partial s}\mathrm{U}(t,s)v=\mathrm{U}(t,s)\mathrm{A}(s)v.
\end{align*}
\end{enumerate}
\end{lem}
\begin{lem}[Proposition 2.1, \cite{WF}]\label{lem2.2}
Let $\left\{\mathrm{A}(t):t\in [0,T]\right\}$ satisfy the condition (P1)-(P4). If $\left\{\mathrm{U}(t,s):0\leq s\leq t\leq T\right\}$ is the linear evolution system generated by the family of operators $\left\{\mathrm{A}(t):t\in [0,T]\right\}$, then $\left\{\mathrm{U}(t,s):0\leq s\leq t\leq T\right\}$ is \emph{a compact operator} whenever $t-s> 0$.
\end{lem}

\subsection{Mild solution}Let us now give the definition of mild solution of the system \eqref{1} and state the assumptions of $f(\cdot,\cdot)$, for which \eqref{1} possesses a mild solution. 
\begin{Def}
	A function $x\in \mathrm{C}([0,T];\mathbb{X})$ is called a \emph{mild solution} of \eqref{1}, for each $0\leq t\leq T$ and $s\in [0,t)$, if it satisfies the following equation
	\begin{align}\label{2}
	x(t)= \mathrm{U}(t,0)x_{0}+\int^{t}_{0}\mathrm{U}(t,s)f(s,x(s))\mathrm{d}s+\int^{t}_{0}\mathrm{U}(t,s)\mathrm{B}u(s)\mathrm{d}s.
	\end{align}
\end{Def}

In order to obtain the unique mild solution of the system \eqref{1}, we need the following assumptions which are sufficient conditions also. 
\begin{Ass}\label{ass2.2}
		The function $f$ and the operator $\mathrm{B}$ satisfies the following assumptions: 
\begin{enumerate}
\item [(A1)] The function $f:[0,T]\times \mathbb{X}\rightarrow \mathbb{X}$ is continuous and there exists a positive constant $K$ such that
\begin{align*}
	\left\|f(t,x)\right\|_{\mathbb{X}}\leq K, \ \text{ for all } \ (t,x)\in [0,T]\times \mathbb{X}.
\end{align*}  
\item [(A2)] The operator $\mathrm{B}: \mathbb{H}\rightarrow \mathbb{X}$ is a bounded linear operator with $\left\|\mathrm{B}\right\|_{\mathcal{L}(\mathbb{H};\mathbb{X})}=N$, $N> 0$.
\end{enumerate}
\end{Ass}
Under the above assumptions, it can be easily seen that 
\begin{align*}
\int_0^T\|f(t,x)\|_{\mathbb{X}}\mathrm{d}t&\leq KT<+\infty, \\
\int_0^T\|\mathrm{B}u(t)\|_{\mathbb{X}}\mathrm{d} t&\leq \|\mathrm{B}\|_{\mathcal{L}(\mathbb{H};\mathbb{X})}\int_0^T\|u(t)\|_{\mathbb{H}}\mathrm{d}t\leq NT^{1/2}\left(\int_0^T\|u(t)\|_{\mathbb{H}}^2\mathrm{d}t\right)^{1/2}<+\infty,
\end{align*}
and hence $f,\mathrm{B}u\in\mathrm{L}^1([0,T];\mathbb{X})$. As discussed in section 5.7, Chapter 5, \cite{P}, we obtain \emph{a unique mild solution} of the system \eqref{1}. Moreover, we prove the existence of such a solution for a particular $u(\cdot)$ (in fact in the feedback form) in the next section.  Let $x(T;x_0,u)$ be the state value of the system \eqref{1} at terminal state $T$, corresponding to the control $u$ and the initial value $x_0$.
\begin{Def}
For $x_0\in\mathbb{X}$, a set $\mathscr{R}_{T}(x_0)$ is called `\emph{reachable set}' of the system \eqref{1}, defined as follows:
\begin{equation*}
  \mathscr{R}_{T}(x_0) =
    \begin{cases}
     x(T)= x(T;x_0,u)\in \mathbb{X}: u\in \mathrm{L}^{2}([0,T];\mathbb{H}), \text{ and }\\
      x(\cdot)\text{ is a mild solution of \eqref{1} with control }u.\\
      
    \end{cases}   
\end{equation*}
\end{Def}
\begin{Def}
	The system \eqref{1} is said to be approximately controllable on the interval $[0,T]$, if $\overline{\mathscr{R}_T(x_0)}=\mathbb{X},$ where $\overline{\mathscr{R}_T(x_0)}$ is closure of ${\mathscr{R}_T(x_0)}$ in $\mathbb{X}$. 
\end{Def}
Let $\mathrm{B}^{*}$, $\mathrm{U}(T,s)^{*}$ denote the adjoint operators of $\mathrm{B}$ and $\mathrm{U}(T,s)$, respectively.  In this paper, we  need the following important assumption also (see \cite{M}) to establish the approximate controllability results for the system \eqref{1}. 
\begin{Ass}\label{ass2.3}
	We assume that 
\begin{enumerate}
	\item [(A3)] for every $h\in\mathbb{X}$, $z_{\lambda}(h)=\lambda(\lambda \mathrm{I}+\Lambda_{T}\mathrm{J})^{-1}(h)$ converges to zero as $\lambda\downarrow 0$ in strong topology, where 
\begin{equation}\label{2.2}
\left\{
\begin{aligned}
L_Tu&:=\int_0^T\mathrm{U}(T,t)\mathrm{B}u(t)\mathrm{d}t,\\
\Lambda_{T}&:=\int^{T}_{0}\mathrm{U}(T,t)\mathrm{B}\mathrm{B}^{*}\mathrm{U}(T,t)^{*}\mathrm{d}t=L_T(L_T)^*,\\
\mathrm{R}(\lambda,\Lambda_{T})&:=(\lambda \mathrm{I}+\Lambda_{T}\mathrm{J})^{-1},\ \lambda > 0,
\end{aligned}
\right.
\end{equation}
and $z_{\lambda}(h)$ is is a solution of the equation \begin{align}\label{24}\lambda z_{\lambda}+\Lambda_T\mathrm{J}[z_{\lambda}]=\lambda h.\end{align}
\end{enumerate}
\end{Ass}
If $\mathbb{X}$ is a separable Hilbert space, then one can define the resolvent operator as $\mathrm{R}(\lambda,\Lambda_{T})=(\lambda \mathrm{I}+\Lambda_{T})^{-1}$. 
Since $\mathbb{X}$ is a separable, reflexive Banach space, from Lemma 2.2, \cite{M}, we know that for every $h\in\mathbb{X}$ and $\lambda>0$, the equation \eqref{24} has a unique solution  $z_{\lambda}(h)=\lambda(\lambda \mathrm{I}+\Lambda_{T}\mathrm{J})^{-1}(h)=\lambda\mathrm{R}(\lambda,\Lambda_T)(h)$ and \begin{align}\label{25}\|z_{\lambda}(h)\|_{\mathbb{X}}=\|\mathrm{J}[z_{\lambda}(h)]\|_{\mathbb{X}'}\leq\|h\|_{\mathbb{X}}.
\end{align}

From Theorem 2.3, \cite{M}, we also obtain that $z_{\lambda}(h)=\lambda(\lambda \mathrm{I}+\Lambda_{T}\mathrm{J})^{-1}(h)$ converges to zero as $\lambda\downarrow 0$ in strong operator topology if and only if $\Lambda_T$ is positive; that is, $\langle x', \Lambda_T x' \rangle =\|(L_T)^*x'\|_{\mathbb{H}}^2>0,$ for all nonzero $x'\in\mathbb{X}'$. 

\section{Linear Non-autonomous Control Problem} \label{sec3}\setcounter{equation}{0}
In this section, we consider the linear problem corresponding to the system \eqref{1}, formulate an optimal control problem and discuss about its connection to the approximate controllability of the linear system. 
\subsection{The optimal control problem for the linear system} In this subsection, we consider a linear regulator problem, consisting of minimizing a cost functional. Our aim is to find the optimal control $u$, which is used in the approximate control system. The cost functional is given by 
\begin{equation}\label{3}
\mathcal{J}(x,u)=\left\|x(T)-x_{T}\right\|^{2}_{\mathbb{X}}+\lambda\int^{T}_{0}\left\|u(t)\right\|^{2}_{\mathbb{H}}\mathrm{d}t,
\end{equation}
where $x(\cdot)$ is the solution of the linear system 
\begin{equation}\label{2.3}
\left\{
\begin{aligned}
\dot{x}(t)&= \mathrm{A}(t)x(t)+\mathrm{B}u(t),\ t\in[0,T],\\
x(0)&=x_{0},
\end{aligned}
\right.
\end{equation} 
with control $u$, $x_{T}\in \mathbb{X}$ and $\lambda >0$. 
	We take the admissible control  class as $$\mathscr{U}_{\text{ad}}=\mathrm{L}^{2}([0,T];\mathbb{H}),$$ consisting of controls $u$. Since $\mathrm{B}u\in\mathrm{L}^1([0,T];\mathbb{X})$, the system \eqref{2.3} has a unique mild solution given by 
	\begin{align}\label{2.4}
	x(t)= \mathrm{U}(t,0)x_{0}+\int^{t}_{0}\mathrm{U}(t,s)\mathrm{B}u(s)\mathrm{d}s,
	\end{align}
	for any $u\in\mathscr{U}_{\text{ad}}$. 
\begin{Def}[Admissible class]\label{definition 1}
	The \emph{admissible class} $\mathscr{A}_{\text{ad}}$ of pairs $(x,u)$ is defined as the set of states $x$ solving the system \eqref{2.3} with control $u \in \mathscr{U}_{ad}$. That is,
	\begin{align*}
	\mathscr{A}_{\text{ad}}:=\Big\{(x,u) :x\text{ is \text{a unique mild solution} of }\eqref{2.3}  \text{ with control }u\in\mathscr{U}_{\text{ad}}\Big\}.
	\end{align*}
\end{Def}
Note that $\mathscr{A}_{\text{ad}}$ is  a nonempty set as for any $u \in \mathscr{U}_{\text{ad}}$, there exists a \emph{unique mild solution} of the system \eqref{2.3}. In view of the above definition, the optimal control problem we are considering can be  formulated as:
\begin{align}\label{4}
\min_{ (x,u) \in \mathscr{A}_{\text{ad}}}  \mathcal{J}(x,u).
\end{align}

	A solution to the problem \eqref{4} is called an \emph{optimal solution}. The optimal pair is denoted by $({x}^0, u^0)$. The control $u^0$ is called an \emph{optimal control}.

\begin{thm}[Existence of an optimal pair]\label{optimal}
	Let  $x_0\in\mathbb{X}$ be given.  Then there exists at least one pair  $(x^0,u^0)\in\mathscr{A}_{\text{ad}}$  such that the functional $\mathcal{J}(x,u)$ attains its minimum at $(x^0,u^0)$, where $x^0$ is the unique mild solution of the system  \eqref{2.3}  with the control $u^0$.
\end{thm}
\begin{proof}
	Let us first define $$\mathcal{J} := \inf \limits _{u \in \mathscr{U}_{\text{ad}}}\mathcal{J}(x,u).$$
	Since, $0\leq \mathcal{J} < +\infty$, there exists a minimizing sequence $\{u^n\} \in \mathscr{U}_{\text{ad}}$ such that $$\lim_{n\to\infty}\mathcal{J}(x^n,u^n) = \mathcal{J},$$ where $x^n(\cdot)$ is the unique mild solution of the system \eqref{2.3} with the control $u^n$ and  the initial data  
$
x^n(0)=	x_0 \in\mathbb{X}.
$
	Note that $x^n(\cdot)$ satisfies
	\begin{align}\label{3.8}
		x^n(t)= \mathrm{U}(t,0)x_{0}+\int^{t}_{0}\mathrm{U}(t,s)\mathrm{B}u^n(s)\mathrm{d}s.
	\end{align}
	Since  $0\in\mathscr{U}_{\text{ad}}$, without loss of generality, we may assume that $\mathcal{J}(x^n,u^n) \leq \mathcal{J}(x,0)$, where $(x,0)\in\mathscr{A}_{\text{ad}}$. Using the definition of $\mathcal{J}(\cdot,\cdot)$, this easily gives
	\begin{align}
	\left\|x^n(T)-x_{T}\right\|^{2}_{\mathbb{X}}+\lambda\int^{T}_{0}\left\|u^n(t)\right\|^{2}_{\mathbb{H}}\mathrm{d}t\leq \left\|x(T)-x_{T}\right\|^{2}_{\mathbb{X}}\leq 2\left(\|x(T)\|_{\mathbb{X}}^2+\|x_T\|_{\mathbb{X}}^2\right)<+\infty,
	\end{align}
	From the above relation, it is clear that, there exist a $R>0$, large enough such that
	$$0 \leq \mathcal{J}(x^n,u^n) \leq R < +\infty.$$
	In particular, there exists a large $\widetilde{C}>0,$ such that
\begin{align}\label{39}\int_0^T \|u^n(t)\|^2_{\mathbb{H}} \mathrm{d} t \leq  \widetilde{C} < +\infty .\end{align}Moreover, from \eqref{3.8}, we have 
	\begin{align}
	\|x^n(t)\|_{\mathbb{X}}&\leq\|\mathrm{U}(t,0)x^0\|_{\mathbb{X}}+\int_0^t\|\mathrm{U}(t,s)\mathrm{B}u^n(s)\|_{\mathbb{X}}\mathrm{d}s \nonumber\\&\leq\|\mathrm{U}(t,0)\|_{\mathcal{L}(\mathbb{X})}\|x^0\|_{\mathbb{X}}+\int_0^t\|\mathrm{U}(t,s)\|_{\mathcal{L}(\mathbb{X})}\|\mathrm{B}\|_{\mathcal{L}(\mathbb{H};\mathbb{X})}\|u^n(s)\|_{\mathbb{H}}\mathrm{d}s\nonumber\\&\leq C\|x^0\|_{\mathbb{X}}+CNt^{1/2}\left(\int_0^t\|u^n(s)\|_{\mathbb{H}}^2\mathrm{d}s\right)^{1/2}\nonumber\\&\leq C\|x^0\|_{\mathbb{X}}+CNt^{1/2}\widetilde{C}^{1/2}<+\infty,
	\end{align}
	for all $t\in[0,T]$. Since  $\mathrm{L}^{2}([0,T];\mathbb{X})$ is reflexive, an application of Banach-Alaoglu theorem yields the existence of a subsequence $\{x^{n_k}\}$ of $\{x^n\}$ such that 
	\begin{align}
	x^{n_k}\xrightharpoonup{w}x^0\ \text{ in }\mathrm{L}^{2}([0,T];\mathbb{X}).
	\end{align}
	From \eqref{39}, we also infer that the sequence $\{u^n\}$ is uniformly bounded in the space $\mathrm{L}^2([0,T];\mathbb{H})$. Since $\mathrm{L}^2([0,T];\mathbb{H})$ is a separable Hilbert space (in fact reflexive), using Banach-Alaoglu theorem, we can extract a subsequence $\{u^{n_k}\}$ of $\{u^n\}$ such that 
	\begin{align*}
	u^{n_k}\xrightharpoonup{w}u^0\ \text{ in }\mathrm{L}^2([0,T];\mathbb{H})=\mathscr{U}_{\text{ad}}. 
	\end{align*}
	Since $\mathrm{B}$ is a bounded linear operator from $\mathbb{H}$ to $\mathbb{X}$, the above convergence also implies
	\begin{align}\label{310}
	\mathrm{B}	u^{n_k}\xrightharpoonup{w}\mathrm{B}u^0\ \text{ in }\mathrm{L}^2([0,T];\mathbb{X}).
	\end{align}
		Note that 
	\begin{align}
&	\left\|\int_0^t\mathrm{U}(t,s)\mathrm{B}u^{n_k}(s)\mathrm{d}s-\int_0^t\mathrm{U}(t,s)\mathrm{B}u^0(s)\mathrm{d}s\right\|_{\mathbb{X}}\to 0, \ \text{ as } k\to\infty, 
	\end{align}
	for all $t\in[0,T]$,	using the weak convergence given in \eqref{310} and strongly continuous property of $\mathrm{U}(\cdot,\cdot)$ (see Lemma \ref{lem4.1} below and Lemma 3.2, Corollary 3.3, Chapter 3, \cite{LY}  also for one parameter family of compact semigroups). Taking weak limit in the equation \eqref{3.8},  we see that the pair $(x^0,u^0)$ satisfies the following system in the weak sense: 
	\begin{equation}\label{3.15}
	\left\{
	\begin{aligned}
	\dot{x}^0(t)&= \mathrm{A}(t)x^0(t)+\mathrm{B}u^0(t),\ t\in[0,T],\\
	x^0(0)&=x_{0},
	\end{aligned}
	\right.
	\end{equation} 
But the existence of a weak solution guarantees the existence of a mild solution (Theorem 1, \cite{JMB}), the system \eqref{3.15}  has a unique mild solution $x\in\mathrm{C}([0,T];\mathbb{X})$ such that 
		\begin{align}\label{3.16}
	x^0(t)= \mathrm{U}(t,0)x_{0}+\int^{t}_{0}\mathrm{U}(t,s)\mathrm{B}u^0(s)\mathrm{d}s.
	\end{align}
	Along a subsequence of \eqref{3.8}, one can easily get 
	\begin{align}\label{316}
	\|x^{n_k}(t)-x^0(t)\|_{\mathbb{X}}=	\left\|\int_0^t\mathrm{U}(t,s)\mathrm{B}u^{n_k}(s)\mathrm{d}s-\int_0^t\mathrm{U}(t,s)\mathrm{B}u^0(s)\mathrm{d}s\right\|_{\mathbb{X}}\to 0, \ \text{ as } k\to\infty, 
	\end{align}
	for all $t\in[0,T]$ and hence $x^{n_k}\to x^0$ in $\mathrm{C}([0,T];\mathbb{X})$, as $k\to\infty$. Since $x^0$ is the unique mild solution of \eqref{3.15}, the whole sequence $\{x_n\}$ converges to  $x^0$. Since $u^0\in\mathscr{U}_{\text{ad}}$ and $x^0$ is the unique mild solution of \eqref{3.15} corresponding to the control $u^0$, it is immediate that $(x^0,u^0)\in\mathscr{A}_{\text{ad}}$.
	
	Let us now show that  $(x^0,u^0)$ is a minimizer, that is, \emph{$\mathcal{J}=\mathcal{J}(x^0,u^0)$}.  Since the cost functional $\mathcal{J}(\cdot,\cdot)$ is continuous and convex (see Proposition III.1.6 and III.1.10,  \cite{EI}) on $\mathrm{L}^2([0,T];\mathbb{X}) \times \mathrm{L}^2([0,T];\mathbb{H})$, it follows that $\mathcal{J}(\cdot,\cdot)$ is weakly lower semi-continuous (Proposition II.4.5, \cite{EI}). That is, for a sequence 
	$$(x^n,u^n)\xrightharpoonup{w}(x^0,u^0)\ \text{ in }\ \mathrm{L}^2([0,T];\mathbb{X}) \times  \mathrm{L}^2([0,T];\mathbb{H}),$$
	we have 
	\begin{align*}
	\mathcal{J}(x^0,u^0) \leq  \liminf \limits _{n\rightarrow \infty} \mathcal{J}(x^n,u^n).
	\end{align*}
	Therefore, we obtain 
	\begin{align*}\mathcal{J} \leq \mathcal{J}(x^0,u^0) \leq  \liminf \limits _{n\rightarrow \infty} \mathcal{J}(x^n,u^n)=  \lim \limits _{n\rightarrow \infty} \mathcal{J}(x^n,u^n) = \mathcal{J},\end{align*}
	and hence $(x^0,u^0)$ is a minimizer of the problem \eqref{4}.
		\end{proof}

\begin{rem}
	Since the cost functional \eqref{3} is convex, the constraint system \eqref{2.3} is linear and $\mathscr{U}_{\text{ad}}=\mathrm{L}^2([0,T];\mathbb{H})$ is convex, optimal control obtained in Theorem \ref{optimal} is unique. 
\end{rem}

Note that an optimal control ${u}$, associated with the approximate controllability of an integer order linear differential equation, is just an optimal solution of the optimal control problem \eqref{4} (see \cite{M2}). Adapting this idea in the following lemma, we find a precise expression of an optimal  control ${u}$, which is given by the feedback law.
\begin{lem}\label{lem3.1}
Assume that ${u}$ is the optimal control satisfying \eqref{1} and minimizing the cost functional \eqref{2}. Then ${u}$ is given by
\begin{align*}
{u}(t)=\mathrm{B}^{*}\mathrm{U}(T,t)^{*}\mathrm{J}\left[\mathrm{R}(\lambda,\Lambda_{T})p(x(\cdot))\right],\ t\in[0,T],
\end{align*}
with
\begin{align*}
p(x(\cdot))=x_{T}-\mathrm{U}(T,0)x_{0}.
\end{align*}
\end{lem}
\begin{proof}
Let $(x,u)$ be an optimal solution of \eqref{4} with the control $u$ and the corresponding trajectory be $x$. Then $\varepsilon=0$ is the critical point of 
\begin{align*}
\mathcal{I}(\varepsilon)=\mathcal{J}(x_{u+\varepsilon w},u+\varepsilon w),
\end{align*}
with $w\in \mathrm{L}^{2}([0,T];\mathbb{H})$, where $x_{u+\varepsilon w}$ is the unique mild solution of \eqref{2.3} with respect to the control $u+\varepsilon w$ and $x_{u+\varepsilon w}(\cdot)$ satisfies: 
\begin{align}\label{5}
x_{u+\varepsilon w}(t)= \mathrm{U}(t,0)x_{0}+\int^{t}_{0}\mathrm{U}(t,s)\mathrm{B}(u+\varepsilon w)(s)\mathrm{d}s.
\end{align}
Let us now compute the variation of the cost functional $\mathcal{J}$ (defined in \eqref{3}) as
\begin{align}\label{320}
\frac{\mathrm{d}}{\mathrm{d}\varepsilon}\mathcal{I}(\varepsilon)\Big|_{\varepsilon=0}&=\frac{\mathrm{d}}{\mathrm{d}\varepsilon}\bigg[\left\|x_{u+\varepsilon w}(T)-x_{T}\right\|^{2}_{\mathbb{X}}+\lambda\int^{T}_{0}\left\|u(t)+\varepsilon w(t)\right\|^{2}_{\mathbb{H}}\mathrm{d}t\bigg]_{\varepsilon=0}\nonumber\\
&=2\bigg[\langle  \mathrm{J}(x_{u+\varepsilon w}(T)-x_{T}), \frac{\mathrm{d}}{\mathrm{d}\varepsilon}(x_{u+\varepsilon w}(T)-x_{T})\rangle\nonumber\\&\qquad +2\lambda\int^{T}_{0}(u(t)+\varepsilon w(t),\frac{\mathrm{d}}{\mathrm{d}\varepsilon}(u(t)+\varepsilon w(t)))\mathrm{d}t\bigg]_{\varepsilon=0}\nonumber\\
&=2\left\langle\mathrm{J}(x(T)-x_T),\int_0^T\mathrm{U}(T,s)\mathrm{B}w(t)\mathrm{d}t \right\rangle+2\lambda\int_0^T(u(t),w(t))\mathrm{d} t. 
\end{align}
Since the first variation of the cost functional is zero, we obtain 
\begin{align}\label{2.7}
0&=\left\langle\mathrm{J}(x(T)-x_T),\int_0^T\mathrm{U}(T,t)\mathrm{B}w(t)\mathrm{d}t \right\rangle+\lambda\int_0^T(u(t),w(t))\mathrm{d} t\nonumber\\&=\int_0^T\left\langle\mathrm{J}(x(T)-x_T),\mathrm{U}(T,t)\mathrm{B}w(t) \right\rangle\mathrm{d}t+\lambda\int_0^T(u(t),w(t))\mathrm{d} t\nonumber\\&= \int_0^T\left(\mathrm{B}^*\mathrm{U}^*(T,t)\mathrm{J}(x(T)-x_T)+\lambda u(t),w(t) \right)\mathrm{d}t.
\end{align}
Since $w\in \mathrm{L}^{2}([0,T];\mathbb{H})$ is an arbitrary element (one can choose $w$ to be $\mathrm{B}^*\mathrm{U}^*(T,t)\mathrm{J}[x(T)-x_T]+\lambda u(t)$), it follows that the optimal control is given by
\begin{align}\label{2.8}
u(t)&= -\lambda^{-1}\mathrm{B}^{*}\mathrm{U}^{*}(T,t)\mathrm{J}[x(T)-x_{T}],\ t\in [0,T],
\end{align}
for a.e. $t\in [0,T]$. It also holds for all $t\in[0,T]$, since from the expressions \eqref{2.7} and \eqref{2.8}, it is clear that $u$ is continuous and belongs to $\mathrm{C}([0,T];\mathbb{X})$. Therefore the state system \eqref{2.3} at a final point $T$ with the above control $u$ is given by
\begin{align}\label{3.9}
x(T)&=\mathrm{U}(T,0)x_{0}-\int^{T}_{0}\lambda^{-1}\mathrm{U}(T,t)\mathrm{B}\mathrm{B}^{*}\mathrm{U}(T,t)^{*}\mathrm{J}\left[x(T)-x_{T}\right]\mathrm{d}s\\
&= \mathrm{U}(T,0)x_{0}-\lambda^{-1}\Lambda_{T} \mathrm{J}\left[x(T)-x_{T}\right].
\end{align}
Let us now define
\begin{align}\label{3.10}
p(x(\cdot)):=x_{T}-\mathrm{U}(T,0)x_{0}.
\end{align}
Combining \eqref{3.9} and \eqref{3.10}, we have the following:
\begin{align}\label{3.12}
x(T)-x_{T}&=-p(x(\cdot))-\lambda^{-1}\Lambda_{T} \mathrm{J}\left[x(T)-x_{T}\right].
\end{align}
From \eqref{3.12}, one can easily deduce that 
\begin{align}\label{3.13}
x(T)-x_T=-\lambda\mathrm{I}(\lambda\mathrm{I}+\Lambda_T\mathrm{J})^{-1}p(x(\cdot))=-\lambda\mathrm{R}(\lambda,\Lambda_T)p(x(\cdot)).
\end{align}
Finally, from \eqref{2.8}, we have 
\begin{align*}
u(t)=\mathrm{B}^{*}\mathrm{U}^{*}(T,t)\mathrm{J}\left[\mathrm{R}(\lambda,\Lambda_{T})p(x(\cdot))\right],\ t\in [0,T],
\end{align*}
which completes the proof. 
\end{proof}

Next, we state and prove the approximate controllability results for the linear non-autonomous system \eqref{2.4}. 

\begin{thm}\label{thm3.2}
	The linear non-autonomous control system \eqref{2.4} is approximately controllable on $[0,T]$ if and only if $\mathrm{ker}\{(L_T)^*\}=0$, where $(L_T)^*$ is defined in \eqref{2.2}.
	\end{thm}
	\begin{proof}
Since the system \eqref{2.3} is linear, $x_0\in\mathbb{X}$ and $\mathrm{B}u\in\mathrm{L}^2([0,T];\mathbb{X})$, we know that for every $\lambda>0$ and $x_{T}\in \mathbb{X}$, there exists a unique mild solution $x_{\lambda}\in \mathrm{C}([0,T];\mathbb{X})$ such that
\begin{align}\label{328}
x_{\lambda}(t)=\mathrm{U}(t,0)x_{0}+\int^{t}_{0}\mathrm{U}(t,s)\mathrm{B}u(s)\mathrm{d}s,\ t\in[0,T],
\end{align}
with
\begin{align*}
u(t)=\mathrm{B}^{*}\mathrm{U}(T,t)^{*}\mathrm{J}[\mathrm{R}(\lambda,\Lambda_{T})p(x_{\lambda}(\cdot))], \ \text{ and }\ p(x_{\lambda}(\cdot))=x_{T}-\mathrm{U}(T,0)x_0. 
\end{align*}
Using \eqref{328}, it can be easily seen that 
\begin{align}\label{3.29}
x_{\lambda}(T)&=\mathrm{U}(T,0)x_{0}+\int^{T}_{0}\mathrm{U}(T,s)\mathrm{B}u(s)\mathrm{d}s\nonumber\\&=\mathrm{U}(T,0)x_{0}+\Lambda_{T}\mathrm{J}[\mathrm{R}(\lambda,\Lambda_{T})p(x_{\lambda}(\cdot))]\nonumber\\
&= x_{T}-p(x_{\lambda}(\cdot))+\Lambda_{T}\mathrm{J}[\mathrm{R}(\lambda,\Lambda_{T})p(x_{\lambda}(\cdot))]\nonumber\\
&= x_{T}-(\lambda \mathrm{I}+\Lambda_{T}\mathrm{J})\mathrm{R}(\lambda,\Lambda_{T})p(x_{\lambda}(\cdot))+\Lambda_{T}\mathrm{J}[\mathrm{R}(\lambda,\Lambda_{T})p(x_{\lambda}(\cdot))]\nonumber\\
&= x_{T}-\lambda \mathrm{R}(\lambda,\Lambda_{T})p(x_{\lambda}(\cdot)),
\end{align}
and since $\|\mathrm{U}(T,0)x_0\|_{\mathbb{X}}\leq C\|x_0\|_{\mathbb{X}}$ and $x_T\in\mathbb{X}$, we have 
\begin{align*}
\left\|x_{\lambda}(T)-x_{T}\right\|_{\mathbb{X}}&\leq\left\|\lambda \mathrm{R}(\lambda,\Lambda_{T})(x_{T}-\mathrm{U}(T,0)x_0)\right\|_{\mathbb{X}}
\rightarrow 0,\ \text{ as}\ \lambda\downarrow 0, 
\end{align*}
if and only if $\langle x',\Lambda_T x'\rangle =\|(L_T)^*x'\|_{\mathbb{H}}^2>0,$ for all non-zero $x'\in\mathbb{X}'$ (Theorem 2.3, \cite{M}). This implies that the  linear  non-autonomous control system \eqref{2.4} is approximately controllable on $[0,T]$.
\end{proof}

\begin{rem}\label{rem3.2}
	1. Note that for $x'\in\mathbb{X}'$ and $u\in\mathrm{L}^2([0,T];\mathbb{H})$, we have 
	\begin{align}
	((L_T)^*x',u)_{\mathrm{L}^2([0,T];\mathbb{H})}&=\langle x',L_Tu\rangle =\left\langle x',\int_0^T\mathrm{U}(T,t)\mathrm{B}u(t)\mathrm{d}t\right\rangle \nonumber\\&=\int_0^T\left\langle x',\mathrm{U}(T,t)\mathrm{B}u(t)\right\rangle\mathrm{d}t\nonumber\\&=\int_0^T(\mathrm{B}^*\mathrm{U}^*(T,t)x',u(t))\mathrm{d}t\nonumber\\&=(\mathrm{B}^*\mathrm{U}^*(T,t)x',u)_{\mathrm{L}^2([0,T];\mathbb{H})}
	\end{align}
	and hence $(L_T)^*=\mathrm{B}^*\mathrm{U}^*(T,t)$. Thus, from Theorem \ref{thm3.2}, it is clear that the system \eqref{2.4} is approximately controllable on $[0,T]$ if and only if $\mathrm{B}^*\mathrm{U}^*(T,t)x'=0$ on  $[0,T]$ implies $x'=0$. 
	
2. For the nonlinear problem, if $u$ appearing in \eqref{1} minimizes the cost functional \eqref{4}, then also one can prove the existence of an optimal control for the problem \eqref{4} in a similar fashion as in Theorem \ref{optimal}. But we need the following assumption of Lipschitz continuity on $f(\cdot,\cdot)$.
\begin{itemize}
\item [(AL)] there exists a positive constant $L$ such that
\begin{align*}
\left\|f(t,x)-f(t,y)\right\|_{\mathbb{X}}\leq L\|x-y\|_{\mathbb{X}},\ \text{ for all } \ (t,x), (t,y)\in [0,T]\times \mathbb{X}.
\end{align*}  
\end{itemize}
 In the proof, one needs to replace \eqref{3.8} with 
\begin{align}\label{4.1}
x^n(t)= \mathrm{U}(t,0)x_{0}+\int_0^t\mathrm{U}(t,s)f(s,x^n(s))\mathrm{d}s+\int^{t}_{0}\mathrm{U}(t,s)\mathrm{B}u^n(s)\mathrm{d}s, 
\end{align}
and \eqref{3.16} with 
\begin{align}\label{4.2}
x^0(t)= \mathrm{U}(t,0)x_{0}+\int_0^t\mathrm{U}(t,s)f(s,x^0(s))\mathrm{d}s+\int^{t}_{0}\mathrm{U}(t,s)\mathrm{B}u^0(s)\mathrm{d}s.
\end{align}
Now, we consider
\begin{align}\label{4.4}
\|x^n(t)-x^0(t)\|_{\mathbb{X}}&\leq\int_0^t\|\mathrm{U}(t,s)\|_{\mathcal{L}(\mathbb{X})}\|f(s,x^n(s))-f(s,x^0(s))\|_{\mathbb{X}}\mathrm{d}s \nonumber\\&\quad+ \left\|\int_0^t\mathrm{U}(t,s)(\mathrm{B}u^{n_k}(s)-\mathrm{B}u^{0}(s))\mathrm{d}s\right\|_{\mathbb{X}}\nonumber\\&\leq CL\int_0^t\|x^n(s)-x^0(s)\|_{\mathbb{X}}\mathrm{d}s+ \left\|\int_0^t\mathrm{U}(t,s)(\mathrm{B}u^{n_k}(s)-\mathrm{B}u^{0}(s))\mathrm{d}s\right\|_{\mathbb{X}}.
\end{align}
An application of Gronwall's inequality in \eqref{4.4} gives 
\begin{align}
\|x^n(t)-x^0(t)\|_{\mathbb{X}}\leq e^{CLt} \left\|\int_0^t\mathrm{U}(t,s)(\mathrm{B}u^{n_k}(s)-\mathrm{B}u^{0}(s))\mathrm{d}s\right\|_{\mathbb{X}}
\to 0, \ \text{ as }\ k\to\infty, 
\end{align}
for all $t\in[0,T]$, using \eqref{316}. 

3. It seems to the authors that obtaining a feedback control which is used to prove the approximate controllability results (see \eqref{41} below) through  optimal control problem technique may not work for nonlinear systems. The difficulty arises in \eqref{320}, when one tries to differentiate the cost functional with respect to $\varepsilon$, as the trajectory $x_{u+\varepsilon w}(\cdot)$ given by
\begin{align}
x_{u+\varepsilon w}(t)=x_0+\int_0^tf(s,x_{u+\varepsilon w}(s))\mathrm{d}s+\int_0^t\mathrm{B}(u+\varepsilon w)(s)\mathrm{d}s,
\end{align}
depends on $\varepsilon$ in a nonlinear fashion. 
\end{rem}

\section{Approximate controllability of the nonlinear system}\label{sec4}\setcounter{equation}{0}
In this section, we show the existence and approximate controllability of the system \eqref{1}.
Motivated from the case of linear system, for every $\lambda>0$ and $x_{T}\in \mathbb{X}$, we consider the following integral system:
\begin{equation}\label{41}
\left\{
\begin{aligned}
x(t) &=\mathrm{U}(t,0)x_{0}+\int^{t}_{0}\mathrm{U}(t,s)\left[f(s,x(s))\right]+\mathrm{B}u(s)\mathrm{d}s, \ 0\leq t\leq T,\\
u(t)&=\mathrm{B}^{*}\mathrm{U}^{*}(T,t)\mathrm{J}\left[\mathrm{R}(\lambda,\Lambda_{T})p(x(\cdot))\right],\\
p(x(\cdot))&=x_{T}-\mathrm{U}(T,0)-\int^{T}_{0}\mathrm{U}(T,s)f(s,x(s))\mathrm{d}s.
\end{aligned} 
\right.
\end{equation}
We show that  the system \eqref{1} is approximately controllable if for all $\lambda>0$, there exists a continuous function  $(x,u)\in\mathrm{C}([0,T];\mathbb{X})\times\mathrm{C}([0,T];\mathbb{H})$. More precisely, we would like to approximate any fixed point $x_{T}\in \mathbb{X}$ under appropriate conditions by using the final state of the solution $x$ with the control $u$ given in system \eqref{41}. With this fact in mind,  our aim in this section is to find conditions for solvability of the system \eqref{41}. In order to do this,  we first show that for every $\lambda> 0$ and $x_{T}\in \mathbb{X}$, the system \eqref{1} has at least one mild solution. Then, for any given any $x_{T}\in \mathbb{X},$  we can approximate it with these solutions $\left\{x_{\lambda}:\lambda>0\right\}$.
Proof of the following lemma is similar to Lemma 3.1, \cite{MI}, and for completeness, we give a proof here. 
\begin{lem}\label{lem4.1}
Let the Assumptions  (P1)-(P4) hold. Let $\mathrm{G}:\mathrm{C}([0,T];\mathbb{X})\rightarrow \mathrm{C}([0,T];\mathbb{X})$ be the \emph{Cauchy operator} defined by 
\begin{align}
(\mathrm{G}\varphi)(t)= \int^{t}_{0}\mathrm{U}(t,s)\varphi(s)\mathrm{d}s, \ t\in [0,T].
\end{align}
Then $\mathrm{G}$ is a compact operator.
\begin{proof}
We prove that $\mathrm{G}$ is a compact operator by making use of the Arzel\'a-Ascoli theorem. Let a closed and bounded ball $\mathcal{B}_R$ in $\mathrm{C}([0,T];\mathbb{X})$ be defined as 
\begin{align}\label{43}
\mathcal{B}_R=\left\{\varphi\in \mathrm{C}([0,T];\mathbb{X}):\left\|\varphi\right\|_{\mathbb{X}}\leq R\right\}.
\end{align}
In order to use the Arzel\'a-Ascoli theorem, we first show that the set $\mathrm{G}\mathcal{B}_R$ is an equicontinuous set on $\mathrm{C}([0,T];\mathbb{X})$. For $\varphi\in \mathcal{B}_R$ and $0\leq t_{1}\leq t_{2}\leq T$, we consider the following:
\begin{align*}
J_1&=\left\|(\mathrm{G}\varphi)(t_{2})-(\mathrm{G}\varphi)(t_{1})\right\|_{\mathbb{X}}\\
&=\left\|\int^{t_{2}}_{0}\mathrm{U}(t_{2},s)\varphi(s)\mathrm{d}s-\int^{t_{1}}_{0}\mathrm{U}(t_{1},s)\varphi(s)\mathrm{d}s\right\|_{\mathbb{X}}\\ &\leq \left\|\int^{t_{1}}_{0}\left(\mathrm{U}(t_{2},s)-\mathrm{U}(t_{1},s)\right)\varphi(s)\mathrm{d}s\right\|_{\mathbb{X}}+\left\|\int_{t_{1}}^{t_2}\mathrm{U}(t_{2},s)\varphi(s)\mathrm{d}s\right\|_{\mathbb{X}}  \\
&\leq \int^{t_{1}}_{0}\left\|(\mathrm{U}(t_{2},s)-\mathrm{U}(t_{1},s))\varphi(s)\right\|_{\mathbb{X}}\mathrm{d}s+\int^{t_{2}}_{t_{1}}\left\|\mathrm{U}(t_{2},s)\right\|_{\mathcal{L}(\mathbb{X})}\|\varphi(s)\|_{\mathbb{X}}\mathrm{d}s\\
&\leq \int^{t_{1}}_{0}\left\|(\mathrm{U}(t_{2},s)-\mathrm{U}(t_{1},s))\varphi(s)\right\|_{\mathbb{X}}\mathrm{d}s+CR(t_{2}-t_{1}).
\end{align*}
For $t_{1}=0$, we have 
$
\lim\limits_{t_{2}\rightarrow 0}J_1=0, \ \text{uniformly for} \ \varphi\in \mathcal{B}_R.
$
If $0< t_{1}< T$, for $0< \delta < t_{1}$, we rewrite $J_1$ as 
\begin{align*}
J_1&\leq \int^{t_{1}-\delta}_{0}\left\|(\mathrm{U}(t_{2},s)-\mathrm{U}(t_{1},s))\varphi(s)\right\|_{\mathbb{X}}\mathrm{d}s+R\int^{t_{1}}_{t_{1}-\delta}\left\|\mathrm{U}(t_{2},s)-\mathrm{U}(t_{1},s)\right\|_{\mathbb{X}}\mathrm{d}s+CR(t_{2}-t_{1})\nonumber\\&=:J_2+J_3+J_4,
\end{align*}
where $J_i$, $i=2,3,4,$ are the terms appearing in the right hand side of the above inequality. Clearly, as $t_{2}\rightarrow t_{1}$ $J_4\to 0$. For sufficiently small   $\delta$, note that the compactness of $\mathrm{U}(t,s)$ for $t-s>0$, implies the continuity in the uniform operator topology and hence $J_3\to 0$ as $t_1\to t_2$.  Since $\mathrm{U}(t,s)$ is strongly continuous, $J_2\to 0$ as $t_1\to t_2$, for all $\varphi\in\mathcal{B}_R\subset\mathbb{X}$. Thus, the right hand side of the above inequality tends to zero, independent of $\varphi\in \mathcal{B}_R$. Since $\delta>0$ is arbitrary, we conclude that
\begin{align*}
\lim_{\left|t_{1}-t_{2}\right|\rightarrow 0}J_1= 0, \ \text{uniformly for} \ \varphi\in \mathcal{B}_R,
\end{align*}
which implies $\mathrm{G}\mathcal{B}_R$ is equicontinuous on $\mathrm{C}([0,T];\mathbb{X})$.

Let us now show that $\left\{(\mathrm{G}\varphi)(t):\varphi\in \mathcal{B}_R\right\}$ is precompact in $\mathbb{X}$, for every $t\in [0,T]$. Let $\mathcal{B}_R$ be the bounded subset of $\mathrm{C}([0,T];\mathbb{X})$, $0<t\leq T$ as defined in \eqref{43}. For $0<t\leq T$ and $0<\varepsilon<t$, we consider the following:
\begin{align*}
&\left\|\mathrm{U}(t,t-\varepsilon)\int^{t-\varepsilon}_{0}\mathrm{U}(t-\varepsilon,s)\varphi(s)\mathrm{d}s-\int^{t}_{0}\mathrm{U}(t,s)\varphi(s)\mathrm{d}s\right\|_{\mathbb{X}}\nonumber\\ &\leq \int_0^{t-\varepsilon}\|(\mathrm{U}(t,t-\varepsilon)\mathrm{U}(t-\varepsilon,s)-\mathrm{U}(t,s))\varphi(s)\|_{\mathbb{X}}\mathrm{d}s+\int_{t-\varepsilon}^t\|\mathrm{U}(t,s)\varphi(s)\|_{\mathbb{X}}\mathrm{d}s \nonumber\\
&\leq R\int^{t-\varepsilon}_{0}\left\|\mathrm{U}(t,t-\varepsilon)\mathrm{U}(t-\varepsilon,s)-\mathrm{U}(t,s)\right\|_{\mathcal{L}(\mathbb{X})}\mathrm{d}s+R\int^{t}_{t-\varepsilon}\left\|\mathrm{U}(t,s)\right\|_{\mathcal{L}(\mathbb{X})}\mathrm{d}s.
\end{align*}
Using the semigroup property of the evolution system $\left\{\mathrm{U}(t,s)\right\}$, where $t-s>0$, the first term on the right hand side of the above inequality is zero and we conclude that
\begin{align*}
\left\|\mathrm{U}(t,t-\varepsilon)\int^{t-\varepsilon}_{0}\mathrm{U}(t-\varepsilon,s)\varphi(s)\mathrm{d}s-\int^{t}_{0}\mathrm{U}(t,s)\varphi(s)\mathrm{d}s\right\|_{\mathbb{X}}\leq CR{\varepsilon}.
\end{align*}
The above expression shows that $\left\{(\mathrm{G}\varphi)(t):\varphi\in \mathcal{B}_R\right\}$ is precompact in $\mathbb{X}$.  Thus, by using the total boundedness and equicontinuity, we conclude that the operator $\mathrm{G}$ is compact in view of Arzel\'a-Ascoli theorem.
\end{proof}
\end{lem}
\begin{thm}\label{thm4.1}
Let the Assumptions (P1)-(P4) and (A1)-(A2) hold true. Then the system \eqref{1} has at least one mild solution on $[0,T],$ for every $\lambda>0$ and for fixed $x_{T}\in \mathbb{X}$.
\begin{proof}
For fixed $\lambda>0$ and given $x_{T}\in \mathbb{X}$, we define the solution operator $\Psi:\mathrm{C}([0,T];\mathbb{X})\rightarrow \mathrm{C}([0,T];\mathbb{X})$ as follows:
\begin{equation}\label{44}
(\Psi x)(t)=\mathrm{U}(t,0)x_{0}+\int^{t}_{0}\mathrm{U}(t,s)\left[f(s,x(s))+\mathrm{B}u(s)\right]\mathrm{d}s,\ t\in [0,T],
\end{equation}
with
\begin{align*}
u(t)=\mathrm{B}^{*}\mathrm{U}(T,t)^{*}\mathrm{J}[\mathrm{R}(\lambda,\Lambda_{T})p(x(\cdot)) ]
\end{align*}
and
\begin{align*}
p(x(\cdot))=x_{T}-\mathrm{U}(T,0)x_{0}-\int^{T}_{0}\mathrm{U}(T,s)f(s,x(s))\mathrm{d}s.
\end{align*}
From the definition of $\Psi$, it is clear that the fixed point of $\Psi$ is a mild solution of the  system \eqref{1}. We obtain the fixed point of $\Psi$ by using \emph{Schauder's fixed point theorem}. 

\vskip 0.1in 
\noindent\textbf{Step (1): }\emph{$\Psi$ is a continuous operator.}
Let us first show that the mapping $\Psi$ is a continuous operator on $\mathrm{C}([0,T];\mathbb{X})$. Let $\left\{x_{n}\right\}_{n\in \mathbb{N}}$ be a sequence in $\mathrm{C}([0,T];\mathbb{X})$ with $\lim\limits_{n\rightarrow \infty}x_{n}=x$ in $\mathrm{C}([0,T];\mathbb{X})$, that is, $$\lim\limits_{n\to\infty}\sup\limits_{t\in[0,T]}\|x^n(t)-x(t)\|_{\mathbb{X}}=0.$$ Remember that the function $f$ is continuous, and using the strong convergence of $x_n\to x$ uniformly, we have $f(s,x_{n}(s))$ converges to $f(s,x(s))$ uniformly for $s\in [0,T]$. Thus, for $s\in[0,T]$, we have
\begin{align}\label{45}
\left\|p(x_{n}(\cdot))-p(x(\cdot))\right\|_{\mathbb{X}}&=\left\|\int^{T}_{0}\mathrm{U}(T,s)\left[f(s,x_n(s))-f(s,x(s))\right]\mathrm{d}s\right\|_{\mathbb{X}}\nonumber\\&\leq\int_0^T\|\mathrm{U}(T,s)\|_{\mathcal{L}(\mathbb{X})}\|f(s,x_n(s))-f(s,x(s))\|_{\mathbb{X}}\mathrm{d}s\nonumber\\&\leq CT \sup_{s\in [0,T]}\left\|f(s,x_{n}(s))-f(s,x(s))\right\|_{\mathbb{X}}\to 0, \ \text{ as }\ n\to\infty.
\end{align}
Using \eqref{25} and \eqref{45}, for $\lambda>0$, we know that 
\begin{align}\label{46}
\|\mathrm{R}(\lambda,\Lambda_{T})p(x_n(\cdot))-\mathrm{R}(\lambda,\Lambda_{T})p(x(\cdot))\|_{\mathbb{X}}&= \frac{1}{\lambda}\|\lambda\mathrm{R}(\lambda,\Lambda_{T})\left(p(x_n(\cdot))-p(x(\cdot))\right)\|_{\mathbb{X}}\nonumber\\&\leq\frac{1}{\lambda}\|p(x_n(\cdot))-p(x(\cdot))\|_{\mathbb{X}}\to 0\ \text{ as }\ \to\infty,
\end{align}
and hence $\mathrm{R}(\lambda,\Lambda_{T})p(x_n(\cdot))\to \mathrm{R}(\lambda,\Lambda_{T})p(x(\cdot))$ in $\mathbb{X}$ as $n\to\infty$. Since the mapping $\mathrm{J}:\mathbb{X}\to\mathbb{X}'$  is  demicontinuous, it is immediate that 
\begin{align}\label{48}
\mathrm{J}\left[\mathrm{R}(\lambda,\Lambda_{T})p(x^n(\cdot))\right]\xrightharpoonup{w}\mathrm{J}\left[\mathrm{R}(\lambda,\Lambda_{T})p(x(\cdot))\right] \ \text{ as } \ n\to\infty  \ \text{ in }\ \mathbb{X}'.
\end{align}
 Remember that product of a compact operator and a bounded linear operator is again a compact operator. Since $\mathrm{U}(t,s)$ is compact for $t>s$ and $\mathrm{B}$ is a bounded linear operator from $\mathbb{H}$ to $\mathbb{X}$, we obtain that $\mathrm{U}(t,s)\mathrm{B}$ is a compact operator for $t>s$ in $\mathbb{X}$. Also, an operator $\mathrm{K}$ is compact if and only if its adjoint $\mathrm{K}^*$ is compact. Thus, $(\mathrm{U}(t,s)\mathrm{B})^*$, $t>s$, is a compact operator on $\mathbb{X}'$ and finally we have $\mathrm{U}(t,s)\mathrm{B}(\mathrm{U}(t,s)\mathrm{B})^*=\mathrm{U}(t,s)\mathrm{B}\mathrm{B}^*\mathrm{U}^*(t,s),$ $t>s$, is a compact operator  on $\mathbb{X}$. As we  proved in Lemma \ref{lem4.1} (see below also), one can show that the operator $$\varphi\mapsto\int_0^t\mathrm{U}(t,s)\mathrm{B}\mathrm{B}^*\mathrm{U}(t,s)\varphi(s)\mathrm{d}s$$ is a compact operator. Combining this fact with  \eqref{48}, it can be easily deduced that 
 \begin{align}\label{49}
 \left\|\int_0^t\mathrm{U}(t,s)\mathrm{B}\mathrm{B}^{*}\mathrm{U}(T,t)^{*}\left\{\mathrm{J}\left[\mathrm{R}(\lambda,\Lambda_{T})p(x^n(\cdot))\right]-\mathrm{J}\left[\mathrm{R}(\lambda,\Lambda_{T})p(x(\cdot))\right]\right\}\mathrm{d}s\right\|_{\mathbb{X}}\to 0 \ \text{ as } \ n\to\infty. 
 \end{align}
Hence, using \eqref{45} and \eqref{49}, for $t\in [0,T]$, we obtain
\begin{align*}
&\left\|(\Psi x_{n})(t)-(\Psi x)(t)\right\|_{\mathbb{X}}\nonumber\\&\leq \int_0^t\|\mathrm{U}(t,s)\|_{\mathcal{L}(\mathbb{X})}\|f(s,x_n(s))-f(s,x(s))\|_{\mathbb{X}}\mathrm{d}s\nonumber\\&\quad +\left\|\int_0^t\mathrm{U}(t,s)\mathrm{B}\mathrm{B}^{*}\mathrm{U}(T,t)^{*}\left\{\mathrm{J}\left[\mathrm{R}(\lambda,\Lambda_{T})p(x^n(\cdot))\right]-\mathrm{J}\left[\mathrm{R}(\lambda,\Lambda_{T})p(x(\cdot))\right]\right\}\mathrm{d}s\right\|_{\mathbb{X}}\nonumber\\&\leq CT\sup_{s\in[0,T]}\|f(s,x_n(s))-f(s,x(s))\|_{\mathbb{X}} \nonumber\\&\quad +\left\|\int_0^t\mathrm{U}(t,s)\mathrm{B}\mathrm{B}^{*}\mathrm{U}(T,t)^{*}\left\{\mathrm{J}\left[\mathrm{R}(\lambda,\Lambda_{T})p(x^n(\cdot))\right]-\mathrm{J}\left[\mathrm{R}(\lambda,\Lambda_{T})p(x(\cdot))\right]\right\}\mathrm{d}s\right\|_{\mathbb{X}}\to 0,
\end{align*}
as $n\rightarrow \infty$ and thus $\Psi$ is continuous on $\mathrm{C}([0,T];\mathbb{X})$.

\vskip 0.1in
\noindent\textbf{Step (2): }\emph{$\Psi$ is a compact operator.}
Our next aim is to show that the operator $\Psi:\mathrm{C}([0,T];\mathbb{X})\rightarrow \mathrm{C}([0,T];\mathbb{X})$ defined by \eqref{44}, is a compact operator. We follow the works \cite{F,MI}, etc to fulfill this goal. In virtue of Lemma \ref{lem4.1}, it sufficient to establish the compactness of $\Psi_{1}:\mathrm{C}([0,T];\mathbb{X})\rightarrow\mathrm{C}([0,T];\mathbb{X})$, defined by 
\begin{align*}
(\Psi_{1}x)(t)&=\int^{t}_{0}\mathrm{U}(t,s)\mathrm{B}u(s)\mathrm{d}s,\ t\in [0,T],
\end{align*}
with
\begin{align*}
u(t)=\mathrm{B}^{*}\mathrm{U}(T,t)^{*}\mathrm{J}[\mathrm{R}(\lambda,\Lambda_{T})p(x(\cdot))],
\end{align*}
and
\begin{align*}
p(x(\cdot))=x_{T}-\mathrm{U}(T,0)x_{0}-\int^{T}_{0}\mathrm{U}(T,s)f(s,x(s))\mathrm{d}s.
\end{align*}
Let us now establish that $\Psi_{1}$ is compact using Arzel\'a-Ascoli theorem. Let $\mathcal{B}_R$ be the ball defined \eqref{43}  be a bounded subset of $\mathrm{C}([0,T];\mathbb{X})$. For $0\leq t_{1}\leq t_{2}\leq T$ and $x\in \mathcal{B}_R$, let us consider the following:
\begin{align}\label{410}
&\left\|(\Psi_{1}x)(t_{2})-(\Psi_{1}x)(t_{1})\right\|_{\mathbb{X}}\nonumber\\&=\left\|\int^{t_1}_{0}\left(\mathrm{U}(t_2,s)-\mathrm{U}(t_1,s)\right)\mathrm{B}u(s)\mathrm{d}s+\int_{t_1}^{t_2}\mathrm{U}(t_2,s)\mathrm{B}u(s)\mathrm{d}s\right\|_{\mathbb{X}}\nonumber\\
&\leq \int^{t_{1}}_{0}\left\|\mathrm{U}(t_{2},s)-\mathrm{U}(t_{1},s)\right\|_{\mathcal{L}(\mathbb{X)}}\left\|\mathrm{B}\|_{\mathcal{L}(\mathbb{H};\mathbb{X})}\|u(s)\right\|_{\mathbb{H}}\mathrm{d}s+\int^{t_{2}}_{t_{1}}\|\mathrm{U}(t_{2},s)\|_{\mathcal{L}(\mathbb{X})}\|\mathrm{B}\|_{\mathcal{L}(\mathbb{H};\mathbb{X})}\|u(s)\|_{\mathbb{H}}\mathrm{d}s.
\end{align}
Next, we estimate $\|u(t)\|_{\mathbb{H}},$ using \eqref{25} as 
\begin{align}\label{411}
\|u(t)\|_{\mathbb{H}}&=\|\mathrm{B}^{*}\mathrm{U}(T,t)^{*}\mathrm{J}[\mathrm{R}(\lambda,\Lambda_{T})p(x(\cdot))]\|_{\mathbb{H}}\nonumber\\&\leq\frac{1}{\lambda}\|\mathrm{B}^*\|_{\mathcal{L}(\mathbb{X}',\mathbb{H})}\|\mathrm{U}(T,t)^*\|_{\mathcal{L}(\mathbb{X}')}\|\mathrm{J}[\lambda\mathrm{R}(\lambda,\Lambda_{T})p(x(\cdot))]\|_{\mathbb{X}'}\nonumber\\&\leq \frac{CN}{\lambda}\|p(x(\cdot))\|_{\mathbb{X}}\nonumber\\&\leq\frac{CN}{\lambda}\left(\|x_T\|_{\mathbb{X}}+\|\mathrm{U}(T,0)\|_{\mathcal{L}(\mathbb{X})}\|x_0\|_{\mathbb{X}}+\int_0^T\|\mathrm{U}(T,s)\|_{\mathcal{L}(\mathbb{X})}\|f(s,x(s))\|_{\mathbb{X}}\mathrm{d}s\right)\nonumber\\&\leq\frac{CN}{\lambda}\left(\|x_T\|_{\mathbb{X}}+C\|x_0\|_{\mathbb{X}}+CKT\right).
\end{align}
Let us take $\widetilde{C}=\left(\|x_T\|_{\mathbb{X}}+C\|x_0\|_{\mathbb{X}}+CKT\right)$, where the constant $K$ is appearing in Assumption \ref{ass2.2}-(A1). Using \eqref{411} in \eqref{410}, we obtain 
\begin{align}
&\left\|(\Psi_{1}x)(t_{2})-(\Psi_{1}x)(t_{1})\right\|_{\mathbb{X}}\leq \frac{CN^2\widetilde{C}}{\lambda}\int_0^{t_1}\left\|\mathrm{U}(t_{2},s)-\mathrm{U}(t_{1},s)\right\|_{\mathcal{L}(\mathbb{X)}}\mathrm{d}s+\frac{C^2N^2\widetilde{C}}{\lambda}(t_2-t_1). 
\end{align}
For $t_{1}=0$, from the above expression, it is immediate that
\begin{align*}
\lim_{t_{2}\rightarrow 0}\left\|(\Psi_{1}x)(t_{2})-(\Psi_{1}x)(t_{1})\right\|_{\mathbb{X}}=0,\ \text{ uniformly for }\ x\in \mathcal{B}_R.
\end{align*} 
For $0<t_{1}<T$, and for $0<\delta< t_{1}$, we infer that  
\begin{align*}
&\left\|(\Psi_{1}x)(t_{2})-(\Psi_{1}x)(t_{1})\right\|_{\mathbb{X}} \nonumber\\ &
\leq  \frac{CN^2\widetilde{C}}{\lambda}\left[\int_0^{\delta}\left\|\mathrm{U}(t_{2},s)-\mathrm{U}(t_{1},s)\right\|_{\mathcal{L}(\mathbb{X)}}\mathrm{d}s+\int_{\delta}^{t_1}\left\|\mathrm{U}(t_{2},s)-\mathrm{U}(t_{1},s)\right\|_{\mathcal{L}(\mathbb{X)}}\mathrm{d}s\right]\nonumber\\&\quad+\frac{C^2N^2\widetilde{C}}{\lambda}(t_2-t_1)  \nonumber\\ &
\leq  \frac{C^2(1+C)N^2\widetilde{C}\delta}{\lambda}+\frac{CN^2\widetilde{C}}{\lambda}\int_{\delta}^{t_1}\left\|\mathrm{U}(t_{2},s)-\mathrm{U}(t_{1},s)\right\|_{\mathcal{L}(\mathbb{X)}}\mathrm{d}s+\frac{C^2N^2\widetilde{C}}{\lambda}(t_2-t_1).
\end{align*}
From Lemma \ref{lem2.2} (see Proposition 2.1, \cite{WF} also), we know that the evolution system $\mathrm{U}(t,s)$ is compact for $t-s>0$ and hence $\mathrm{U}(t,s)$ is continuous uniformly in an operator norm for $\delta \leq s <t\leq T$ (see Theorem 3.2, Chapter 2, \cite{P}). Thus, using the compactness of $\mathrm{U}(t,s)$ and the arbitrariness of $\delta$ in the above inequality gives
\begin{align*}
\lim_{t_{2}\rightarrow t_{1}}\left\|(\Psi_{1}x)(t_{2})-(\Psi_{1}x)(t_{1})\right\|_{\mathbb{X}} =0, \ \text{ uniformly for }\ x\in \mathcal{B}_R.
\end{align*}
The above expression says that $\Psi_{1}\mathcal{B}_R$ is equicontinuous on $\mathrm{C}([0,T];\mathbb{X})$.

For $t=0$, it is clear that the set $\left\{(\Psi_{1}x)(0):x\in \mathcal{B}_R\right\}$ is precompact in $\mathbb{X}$. Now, for a given $t\in(0,T]$, let us take  $0<\varepsilon <t$. Once again invoking Lemma \ref{lem2.2}, we get that the operator $\mathrm{U}(t,t-\varepsilon)$ is compact. Thus, we have the set 
\begin{align*}
\left\{\mathrm{U}(t,t-\varepsilon)\int^{t-\varepsilon}_{0}\mathrm{U}(t-\varepsilon,s)\mathrm{B}u(s)\mathrm{d}s:x\in \mathcal{B}_R\right\}
\end{align*}
is precompact in $\mathbb{X}$. For $\varepsilon<\delta<T$, we further have
\begin{align}\label{413}
&\left\|\mathrm{U}(t,t-\varepsilon)\int^{t-\varepsilon}_{0}\mathrm{U}(t-\varepsilon,s)\mathrm{B}u(s)\mathrm{d}s-\int^{t-\varepsilon}_{0}\mathrm{U}(t,s)\mathrm{B}u(s)\mathrm{d}s\right\|_{\mathbb{X}}\nonumber\\& =  \int^{t-\varepsilon}_{0}\left\|(\mathrm{U}(t,t-\varepsilon)\mathrm{U}(t-\varepsilon,s)-\mathrm{U}(t,s))\mathrm{B}u(s)\right\|_{\mathbb{X}}\mathrm{d}s\nonumber\\&
\leq\frac{CN^2\widetilde{C}}{\lambda}\int^{t-\delta}_{0}\left\|\mathrm{U}(t,t-\varepsilon)\mathrm{U}(t-\varepsilon,s)-\mathrm{U}(t,s)\right\|_{\mathcal{L}(\mathbb{X})}\mathrm{d}s\nonumber\\&\quad
+\frac{CN^2\widetilde{C}}{\lambda}\int^{t-\varepsilon}_{t-\delta}\left\|\mathrm{U}(t,t-\varepsilon)\mathrm{U}(t-\varepsilon,s)-\mathrm{U}(t,s)\right\|_{\mathcal{L}(\mathbb{X})}\mathrm{d}s\nonumber\\&
\leq\frac{CN^2\widetilde{C}}{\lambda}\int^{t-\delta}_{0}\left\|\mathrm{U}(t,t-\varepsilon)\mathrm{U}(t-\varepsilon,s)-\mathrm{U}(t,s)\right\|_{\mathcal{L}(\mathbb{X})}\mathrm{d}s+\frac{\delta C^2(1+C)N^2\widetilde{C}}{\lambda}.
\end{align} 
It should be noted that $-\varepsilon>-\delta$, and hence $t-\varepsilon>t-\delta$; which also guarantees  the existence of the second term in the right hand side of the inequality \eqref{413} makes sense. Using the semigroup property of the evolution operator $\left\{\mathrm{U}(t,s):t\geq s\right\}$, one can easily see that the first term in the  integral on the right hand side of the inequality \eqref{413} is zero. Thus, the arbitrariness of $\delta$ and the Lebesgue's dominated convergence theorem imply
\begin{align}\label{414}
\lim_{\varepsilon\rightarrow 0}\left\|\mathrm{U}(t,t-\varepsilon)\int^{t-\varepsilon}_{0}\mathrm{U}(t-\varepsilon,s)\mathrm{B}u(s)\mathrm{d}s-\int^{t-\varepsilon}_{0}\mathrm{U}(t,s)\mathrm{B}u(s)\mathrm{d}s\right\|_{\mathbb{X}}=0.
\end{align}
Let us now consider 
\begin{align*}
&\left\|\mathrm{U}(t,t-\varepsilon)\int^{t-\varepsilon}_{0}\mathrm{U}(t-\varepsilon,s)\mathrm{B}u(s)\mathrm{d}s-\int^{t}_{0}\mathrm{U}(t,s)\mathrm{B}u(s)\mathrm{d}s\right\|_{\mathbb{X}}\nonumber\\
&\leq\left\|\mathrm{U}(t,t-\varepsilon)\int^{t-\varepsilon}_{0}\mathrm{U}(t-\varepsilon,s)\mathrm{B}u(s)\mathrm{d}s-\int^{t-\varepsilon}_{0}\mathrm{U}(t,s)\mathrm{B}u(s)\mathrm{d}s\right\|_{\mathbb{X}}\nonumber\\&
\quad+\left\|\int^{t}_{t-\varepsilon}\mathrm{U}(t,s)\mathrm{B}u(s)\mathrm{d}s\right\|_{\mathbb{X}}\nonumber\\& 
\leq\frac{C\varepsilon N^2\widetilde{C}}{\lambda}+\left\|\mathrm{U}(t,t-\varepsilon)\int^{t-\varepsilon}_{0}\mathrm{U}(t-\varepsilon,s)\mathrm{B}u(s)\mathrm{d}s-\int^{t-\varepsilon}_{0}\mathrm{U}(t,s)\mathrm{B}u(s)\mathrm{d}s\right\|_{\mathbb{X}}.
\end{align*}
Using \eqref{414}, we easily have 
\begin{align*}
\lim_{\varepsilon\rightarrow 0}\left\|\mathrm{U}(t,t-\varepsilon)\int^{t-\varepsilon}_{0}\mathrm{U}(t-\varepsilon,s)\mathrm{B}u(s)\mathrm{d}s-\int^{t}_{0}\mathrm{U}(t,s)\mathrm{B}u(s)\mathrm{d}s\right\|_{\mathbb{X}}=0.
\end{align*}
This shows that $\left\{(\Psi_{1}x)(t):x\in \mathcal{B}_R\right\}$ is precompact in $\mathbb{X}$, using the total boundedness. Thus, an application of Arzel\'a-Ascoli theorem yields  that $\Psi_{1}$ is compact.

\vskip 0.1in
\noindent\textbf{Step (3): }\emph{Application of Schauder's fixed point theorem.}
It is left to show that there exists an $R_{0}>0$ such that $\Psi \mathcal{B}_{R_{0}}\subseteq \mathcal{B}_{R_{0}}$. Remember that for all $x\in \mathrm{C}([0,T];\mathbb{X})$, we have 
\begin{align*}
\|(\Psi x)(t)\|_{\mathbb{X}}&\leq\|\mathrm{U}(t,0)x_{0}\|_{\mathbb{X}}+\left\|\int^{t}_{0}\mathrm{U}(t,s)f(s,x(s))\mathrm{d}s\right\|_{\mathbb{X}}+\left\|\int^{t}_{0}\mathrm{U}(t,s)\mathrm{B}u(s)\mathrm{d}s\right\|_{\mathbb{X}}\nonumber\\&\leq\|\mathrm{U}(t,0)\|_{\mathcal{L}(\mathbb{X})}\|x_{0}\|_{\mathbb{X}}+\int_0^t\|\mathrm{U}(t,s)\|_{\mathcal{L}(\mathbb{X})}\|f(s,x(s))\|_{\mathbb{X}}\mathrm{d}s\nonumber\\&\quad+\int_0^t\|\mathrm{U}(t,s)\|_{\mathcal{L}(\mathbb{X})}\|\mathrm{B}\|_{\mathcal{L}(\mathbb{H};\mathbb{X})}\|u(s)\|_{\mathbb{H}}\mathrm{d}s \nonumber\\&\leq C\|x_0\|_{\mathbb{X}}+CKt+\frac{C^2N^2\widetilde{C}t}{\lambda},
\end{align*}
where we used \eqref{411}. Thus, it is immediate that 
\begin{align}\label{4.14}
\sup\limits_{t\in[0,T]}\|(\Psi x)(t)\|_{\mathbb{X}}\leq C\|x_0\|_{\mathbb{X}}+CKT+\frac{C^2N^2\widetilde{C}T}{\lambda}.
\end{align}
From the  inequality \eqref{4.14}, one can easily see that for large enough $R_{0}> 0$, the inequality $\sup\limits_{t\in[0,T]}\|(\Psi x)(t)\|_{\mathbb{X}}\leq R_0$ holds for all $x\in \mathrm{C}([0,T];\mathbb{X})$ and hence  $\Psi \mathcal{B}_{R_{0}}\subseteq \mathcal{B}_{R_{0}}$. Therefore, using Schauder's fixed point theorem, the operator $\Psi$ has a fixed point in $\mathcal{B}_{R_{0}}$, which is a mild solution of the system \eqref{1}.
\end{proof}
\end{thm}
\begin{rem}
	If $\mathbb{X}'$ is uniformly convex (or if $\mathbb{X}$ is uniformly smooth), then the duality mapping $\mathrm{J}:\mathbb{X}\to\mathbb{X}'$ is is uniformly continuous on every bounded subset of $\mathbb{X}$. Then, with the help of \eqref{46}, one can replace the weak continuity given in \eqref{48} with uniform continuity in $\mathcal{B}_{R_0}$ and \eqref{49} can be estimated as
	\begin{align}
&	\left\|\int_0^t\mathrm{U}(t,s)\mathrm{B}\mathrm{B}^{*}\mathrm{U}(T,t)^{*}\left\{\mathrm{J}\left[\mathrm{R}(\lambda,\Lambda_{T})p(x^n(\cdot))\right]-\mathrm{J}\left[\mathrm{R}(\lambda,\Lambda_{T})p(x(\cdot))\right]\right\}\mathrm{d}s\right\|_{\mathbb{X}}\nonumber\\&\leq \int_0^t\|\mathrm{U}(t,s)\|_{\mathcal{L}(\mathbb{X})}\|\mathrm{B}\|_{\mathcal{L}(\mathbb{H};\mathbb{X})}\|\mathrm{B}^*\|_{\mathcal{L}(\mathbb{X}',\mathbb{H})}\|\mathrm{U}(T,t)^*\|_{\mathcal{L}(\mathbb{X}')}\nonumber\\&\qquad\times\|\mathrm{J}\left[\mathrm{R}(\lambda,\Lambda_{T})p(x^n(\cdot))\right]-\mathrm{J}\left[\mathrm{R}(\lambda,\Lambda_{T})p(x(\cdot))\right]\|_{\mathbb{X}'}\mathrm{d}s\nonumber\\&\leq C^2N^2t\|\mathrm{J}\left[\mathrm{R}(\lambda,\Lambda_{T})p(x^n(\cdot))\right]-\mathrm{J}\left[\mathrm{R}(\lambda,\Lambda_{T})p(x(\cdot))\right]\|_{\mathbb{X}'}\nonumber\\&\to 0 \ \text{ as } \ n\to\infty,
	\end{align}
	using the uniform continuity of $\mathrm{J}[\cdot]$. 
\end{rem}
Let us now establish our main result on the approximate controllability of the system \eqref{1}. 
\begin{thm}\label{thm4.2}
Let the Assumptions (P1)-(P4) and (A1)-(A3) hold true. Then the non-autonomous control system \eqref{1} is approximately controllable on $[0,T]$.
\end{thm}
\begin{proof}
Invoking Theorem \ref{lem4.1}, we know that for every $\lambda>0$ and $x_{T}\in \mathbb{X}$, there exists a mild solution $x_{\lambda}\in \mathrm{C}([0,T];\mathbb{X})$ such that
\begin{align}\label{4p14}
x_{\lambda}(t)=\mathrm{U}(t,0)x_{0}+\int^{t}_{0}\mathrm{U}(t,s)\left[f(s,x_{\lambda}(s))+\mathrm{B}u(s)\right]\mathrm{d}s,\ t\in[0,T],
\end{align}
with
\begin{align*}
u(t)=\mathrm{B}^{*}\mathrm{U}(T,t)^{*}\mathrm{J}[\mathrm{R}(\lambda,\Lambda_{T})p(x_{\lambda}(\cdot))],
\end{align*}
and 
\begin{align*}
p(x_{\lambda}(\cdot))=x_{T}-\mathrm{U}(T,0)x_0-\int^{T}_{0}\mathrm{U}(T,s)f(s,x_{\lambda}(s))\mathrm{d}s.
\end{align*}
Using \eqref{4p14}, we have 
\begin{align}\label{417}
x_{\lambda}(T)&=\mathrm{U}(T,0)x_{0}+\int^{T}_{0}\mathrm{U}(T,s)\left[f(s,x_{\lambda}(s))+\mathrm{B}u(s)\right]\mathrm{d}s\nonumber\\&=\mathrm{U}(T,0)x_{0}+\int^{T}_{0}\mathrm{U}(T,s)f(s,x_{\lambda}(s))\mathrm{d}s+\Lambda_{T}\mathrm{J}[\mathrm{R}(\lambda,\Lambda_{T})p(x_{\lambda}(\cdot))]\nonumber\\
&= x_{T}-p(x_{\lambda}(\cdot))+\Lambda_{T}\mathrm{J}[\mathrm{R}(\lambda,\Lambda_{T})p(x_{\lambda}(\cdot))]\nonumber\\
&= x_{T}-(\lambda \mathrm{I}+\Lambda_{T}\mathrm{J})\mathrm{R}(\lambda,\Lambda_{T})p(x_{\lambda}(\cdot))+\Lambda_{T}\mathrm{J}[\mathrm{R}(\lambda,\Lambda_{T})p(x_{\lambda}(\cdot))]\nonumber\\
&= x_{T}-\lambda \mathrm{R}(\lambda,\Lambda_{T})p(x_{\lambda}(\cdot)).
\end{align}
Applying Assumption (A1),  we find 
\begin{align*}
\int^{T}_{0}\left\|f(s,x_{\lambda}(s))\right\|_{\mathbb{X}}^{2}\mathrm{d}s \leq K^{2}T.
\end{align*}
That is, the sequence $\left\{f(\cdot,x_{\lambda}(s)):\lambda>0\right\}$ is a bounded sequence in the reflexive Banach space $\mathrm{L}^{2}([0,T];\mathbb{X})$. Thanks to the Banach-Alaoglu theorem,  we can find a subsequence of $\left\{f(\cdot,x_{\lambda_k}(s)):\lambda_k>0\right\}$ of $\left\{f(\cdot,x_{\lambda}(s)):\lambda>0\right\}$ such that $\left\{f(\cdot,x_{\lambda_k}(s)):\lambda_k>0\right\}$  converging weakly to  $g(\cdot)\in \mathrm{L}^{2}([0,T];\mathbb{X})$. For notational convenience, we use the same index for subsequence also. Let us now define 
\begin{align*}
\eta:=x_{T}-\mathrm{U}(T,0)-\int^{T}_{0}\mathrm{U}(T,s)g(s)\mathrm{d}s.
\end{align*}
Then,  we have 
\begin{align}\label{418}
\left\|p(x_{\lambda}(\cdot))-\eta\right\|_{\mathbb{X}}\leq\left\|\int^{T}_{0}\mathrm{U}(T,s)\left[f(s,x_{\lambda}(s))-g(s)\right]\mathrm{d}s\right\|_{\mathbb{X}}.
\end{align}
As we have proved in Theorem \ref{thm4.1} (see Lemma \ref{lem4.1} also), using the compactness of $\mathrm{U}(t,s)$, one can show that
\begin{align*}
x(t)\mapsto\int^{t}_{0}\mathrm{U}(t,s)x(s)\mathrm{d}s
\end{align*}
from $\mathrm{L}^{2}([0,T];\mathbb{X})$ to $\mathrm{C}([0,T];\mathbb{X})$ is compact (see Lemma 3.2, Chapter 3, \cite{LY}  for one parameter family of compact semigroups). That is, the Cauchy operator $\mathrm{G}:\mathrm{L}^{2}([0,T];\mathbb{X})\rightarrow \mathrm{C}([0,T];\mathbb{X})$ is also compact and since $f(\cdot,x_{\lambda}(\cdot))\xrightharpoonup{w} g(\cdot)$  in $\mathrm{L}^{2}([0,T];\mathbb{X})$, we deduce that 
\begin{align*}
\int^{t}_{0}\mathrm{U}(t,s)\left[f(s,x_{\lambda}(s))-g(s)\right]\mathrm{d}s\rightarrow 0\ \text{ as }\ \lambda\downarrow 0.
\end{align*}
From \eqref{418}, we further obtain 
\begin{eqnarray}\label{419}
\left\|p(x_{\lambda})-\eta\right\|_{\mathbb{X}}\rightarrow 0, \ \text{ as }\ \lambda\downarrow 0.
\end{eqnarray}
Combing \eqref{417}- \eqref{419} and then using Assumption \ref{ass2.2}-(A3) and \eqref{25}, we finally have 
\begin{align*}
\left\|x_{\lambda}(T)-x_{T}\right\|_{\mathbb{X}}&=\left\|\lambda \mathrm{R}(\lambda,\Lambda_{T})p(x_{\lambda}(\cdot))\right\|_{\mathbb{X}}\\
&\leq \left\|\lambda \mathrm{R}(\lambda,\Lambda_{T})(p(x_{\lambda}(\cdot))-\eta)\right\|_{\mathbb{X}}+\left\|\lambda \mathrm{R}(\lambda,\Lambda_{T})\eta\right\|_{\mathbb{X}}\\  &\leq \left\|p(x_{\lambda}(\cdot))-\eta\right\|_{\mathbb{X}}+\left\|\lambda \mathrm{R}(\lambda,\Lambda_{T})\eta\right\|_{\mathbb{X}}\\
&\rightarrow 0,\ \text{ as}\ \lambda\downarrow 0, 
\end{align*}
which implies that the non-autonomous control system \eqref{1} is approximately controllable on $[0,T]$.
\end{proof}

\section{Application}\label{sec5}\setcounter{equation}{0}
 Let us now provide an example of nonlinear non-autonomous diffusion control system to validate the theory we developed in sections \ref{sec3} and \ref{sec4}. We first consider the one dimensional Laplace operator and discuss its properties. 

\subsection{One dimensional Laplace operator}
Let $p\geq 2$ and $\mathbb{X}=\mathrm{L}^p([0,\pi];\mathbb{R})$. Note that for $1<p<\infty$, $\mathbb{X}$ is a reflexive Banach space. We consider the following operator:
\begin{equation}\label{420}
\left\{
\begin{aligned}
\mathrm{A}f(\xi)&=f''(\xi), \ \text{ a.e. }\ x\in(0,\pi),\\
\mathrm{D}(\mathrm{A})&=\mathrm{W}^{2,p}([0,\pi];\mathbb{R})\cap\mathrm{W}^{1,p}_0([0,\pi];\mathbb{R}).
\end{aligned}\right.
\end{equation}
Note that $\mathrm{C}_0^{\infty}([0,\pi];\mathbb{R})\subset\mathrm{D}(\mathrm{A})$ and hence $\mathrm{D}(\mathrm{A})$ is dense in $\mathbb{X}$. Moreover, $\mathrm{A}$ is closed. Using Hille-Yosida theorem, one can show that $\mathrm{A}$ generates a \emph{$\mathrm{C}_0$-semigroup of contractions} $\mathrm{T}(t)$ on $\mathbb{X}$ (see \eqref{4.24} below also). 

The resolvent set $\rho(\mathrm{A})$ of $\mathrm{A}$ is the set of all complex numbers $\lambda$ for which $\lambda\mathrm{I}-\mathrm{A}$ is invertible, that is, $(\lambda\mathrm{I}-\mathrm{A})^{-1}$ is a bounded linear operator in $\mathbb{X}$.  The family $\mathrm{R}(\lambda,\mathrm{A})=(\lambda\mathrm{I}-\mathrm{A})^{-1}$,  $\lambda\in\rho(\mathrm{A})$ of bounded linear operators is called the resolvent of $\mathrm{A}$.  Let us denote $\sigma(\mathrm{A})=\mathbb{C}\backslash\rho(\mathrm{A})$ as the spectrum of $\mathrm{A}$ (see \cite{P} for more details).  The spectrum of the operator $\mathrm{A}$ defined in \eqref{420} is given by $\sigma(\mathrm{A})=\{-n^2:n\in\mathbb{N}\}$. Let us now show that, for all $\lambda\neq-n^2,n\in\mathbb{N}$, the  Sturm-Liouville system:
\begin{equation}\label{421}
\left\{
\begin{aligned}
\lambda f(\xi)-f''(\xi)&=g(\xi), \ 0<\xi<\pi,\\
f(0)=f(\pi)&=0, 
\end{aligned}
\right.
\end{equation}
has a unique solution for $f\in\mathrm{D}(\mathrm{A})$. Let us use the Fourier series of $g$ to write $g$ as $g(\xi)=\sum\limits_{n=1}^{\infty}g_n\sin(n\xi),$ $\xi\in[0,\pi]$. We seek a solution of the form $f(\xi)=\sum\limits_{n=1}^{\infty}f_n\sin(n\xi)$, $\xi\in[0,\pi]$. For \eqref{421} to be satisfied, we must have $(\lambda+n^2)f_n=g_n$, for all $n\geq 1$. Thus, for any $\lambda\neq-n^2,n\in\mathbb{N}$, the unique solution of \eqref{421} is given by $f(\xi)=\mathrm{R}(\lambda,\mathrm{A})g=(\lambda\mathrm{I}-\mathrm{A})^{-1}g=\sum\limits_{n=1}^{\infty}\frac{g_n}{\lambda+n^2}\sin(n\xi)$, $\xi\in[0,\pi]$. From the expression for $f$ it is also true that $f\in\mathrm{H}^2([0,\pi];\mathbb{R})\cap\mathrm{H}_0^1([0,\pi];\mathbb{R})$. For $ p = 2$, the operator $\mathrm{A}$ in\eqref{420} is self-adjoint and dissipative. Since $\mathrm{A}$ generates a strongly continuous contraction semigroup, and one can show that $\mathrm{A}$ is analytic also. Multiplying both sides of \eqref{421} by $f|f|^{p-2}$ and then integrating over $[0,\pi]$, we find 
\begin{align}\label{422}
\lambda\int_0^{\pi}|f(\xi)|^p\mathrm{d}\xi+(p-1)\int_0^{\pi}|f(\xi)|^{p-2}|f'(\xi)|^2\mathrm{d}\xi=\int_0^{\pi}g(\xi)f(\xi)|f(\xi)|^{p-2}\mathrm{d}\xi,
\end{align}
where we performed an integration by parts also. From the above relation, it is also clear that 
\begin{align}
\lambda\int_0^{\pi}|f(\xi)|^p\mathrm{d}\xi\leq\left(\int_0^{\pi}|g(\xi)|^p\mathrm{d}\xi\right)^{1/p}\left(\int_0^t|f(\xi)|^p\mathrm{d}\xi\right)^{\frac{p-1}{p}},
\end{align}
where we used Holder's inequality. Thus, we have \begin{align}\label{4.24}\|\mathrm{R}(\lambda,\mathrm{A})g\|_{\mathrm{L}^p}=\|f\|_{\mathrm{L}^p}\leq\frac{1}{\lambda}\|g\|_{\mathrm{L}^p}, \ \text{for all } \ \lambda>0.\end{align} For complex values of $\lambda$, one can obtain similar estimates by multiplying \eqref{421} with $\overline{f}|f|^{p-2}$, integrating by parts over $(0,\pi)$ and then considering real and imaginary parts separately and finally combining them together. This ensures that the corresponding semigroup is analytic for $p>2$ also. 

Now, we show that the semigroup $\mathrm{T}(t)$ is compact for $p=2$. Since $\mathrm{A}$ is the infinitesimal generator of a $\mathrm{C}_0$-semigroup of contractions on $\mathrm{L}^p([0,\pi];\mathbb{R})$, in order to prove $\mathrm{T}(t)$ is compact, it is enough to show that the resolvent $\mathrm{R}(\lambda,\mathrm{A})\in\mathcal{K}(\mathbb{X}),$ for some $\lambda\in\rho(\mathrm{A})$. Taking $\lambda=1$ and $p=2$ in \eqref{422}, we have \begin{align}\label{424}\int_0^{\pi}|f'(\xi)|^2\mathrm{d}\xi&\leq\int_0^{\pi}g(\xi)f(\xi)\mathrm{d}\xi\leq\left(\int_0^{\pi}|g(\xi)|^2\mathrm{d}\xi\right)^{1/2}\left(\int_0^{\pi}|f(\xi)|^2\mathrm{d}\xi\right)^{1/2}\nonumber\\&\leq \left(\int_0^{\pi}|g(\xi)|^2\mathrm{d}\xi\right)^{1/2}\left(\int_0^{\pi}|f'(\xi)|^2\mathrm{d}\xi\right)^{1/2},\end{align} where we used Cauchy-Schwarz and Poincar\'e inequalities. Hence, from \eqref{424}, we further have $\|\mathrm{R}(1,\mathrm{A})g'\|_{\mathrm{L}^2}=\|f'\|_{\mathrm{L}^2}\leq\|g\|_{\mathrm{L}^2}$. That is, $\mathrm{R}(1,\mathrm{A})$ maps the unit ball of $\mathrm{L}^2([0,\pi];\mathbb{R})$ into the unit ball of $\mathrm{H}_0^1([0,\pi];\mathbb{R})$, which is compactly embedded in $\mathrm{L}^2([0,\pi];\mathbb{R})$  (Theorem 9.16, \cite{HB}). Let $\{g_n\}_{n\in\mathbb{N}}\in\mathrm{L}^2([0,\pi];\mathbb{R}),$ be a bounded sequence such that $\|g_n\|_{\mathrm{L}^2}\leq K$. Thus, we have  $\|\mathrm{R}(1,\mathrm{A})g_n'\|_{\mathrm{L}^2}\leq\|g_n\|_{\mathrm{L}^2}\leq K$ and hence $\{\mathrm{R}(1,\mathrm{A})g_n\}_{n\in\mathbb{N}}\in\mathrm{H}^1_0([0,\pi];\mathbb{R}),$ is uniformly bounded. Using Morrey's inequality (Theorem 9.12, \cite{HB}), we know that the embedding of $\mathrm{H}_0^1([0,\pi];\mathbb{R})\subset\mathrm{C}^{0,1/2}([0,\pi];\mathbb{R})$ is continuous with 
$$|\mathrm{R}(1,\mathrm{A})g_n(\xi)|\leq \|\mathrm{R}(1,\mathrm{A})g_n\|_{\mathrm{C}^{0,1/2}}\leq\|\mathrm{R}(1,\mathrm{A})g_n'\|_{\mathrm{L}^2}\leq K,$$ for all $x\in[0,\pi]$
and 
$$|\mathrm{R}(1,\mathrm{A})g_n(\xi)-\mathrm{R}(1,\mathrm{A})g_n(\zeta)|\leq C \|\mathrm{R}(1,\mathrm{A})g_n'\|_{\mathrm{L}^2}|\xi-\zeta|^{1/2}\leq CK|\xi-\zeta|^{1/2}.$$ 
 Making use of the Arzel\'a-Ascoli theorem, we can find a subsequence $\{\mathrm{R}(1,\mathrm{A})g_n\}_{n\in\mathbb{N}}$ that converges uniformly. Thus, we conclude that $\mathrm{R}(1,\mathrm{A})$ is compact and hence the semigroup $\mathrm{T}(t)$ is also compact.  For $p>2$, let us consider a unit ball $\mathcal{B}$ in $\mathrm{L}^p([0,\pi];\mathbb{R})$. Then, we have 
 \begin{align*}
\|\mathrm{R}(1,\mathrm{A})g'\|_{\mathrm{L}^2}\leq\|g\|_{\mathrm{L}^2}\leq \pi^{\frac{1}{2}-\frac{1}{p}}\|g\|_{\mathrm{L}^p}\leq \pi^{\frac{1}{2}-\frac{1}{p}},
 \end{align*}
 since $g\in\mathcal{B}$. 
 That is, $\mathrm{R}(1,\mathrm{A})$ maps the unit ball of $\mathrm{L}^p([0,\pi];\mathbb{R})$ into the ball $$\mathcal{B}_r=\left\{g\in\mathrm{H}^{1}_0([0,\pi];\mathbb{R}):\|g'\|_{\mathrm{L}^2}\leq r=: \pi^{\frac{1}{2}-\frac{1}{p}}\right\}$$ of $\mathrm{H}_0^{1}([0,\pi];\mathbb{R})\subset\subset\mathrm{C}([0,\pi];\mathbb{R})\subset\mathrm{L}^p([0,\pi];\mathbb{R}),$ for any $p\in[2,\infty)$  (Theorem 8.8, \cite{HB}). Arguing similarly, as in the case of $p=2$, we obtain that the semigroup $\mathrm{T}(t)$ is compact on $\mathrm{L}^p([0,\pi];\mathbb{R}),$ for $p>2$. 
 Hence the analytic and compact semigroup $\mathrm{T}(t)$ generated by $\mathrm{A}$ can be obtained explicitly as 
 \begin{align*}
 \mathrm{T}(t)x=\sum\limits_{n=1}^{\infty}e^{-n^2 t}\langle x,w_n\rangle w_n, \ \text{ where } \ \langle x,w_n\rangle =\int_0^{\pi}x(\xi)w_n(\xi)\mathrm{d}\xi\ \text{ and }\ w_n(\xi)=\sqrt{\frac{2}{\pi}}\sin(n\xi),
 \end{align*}
 are the normalized eigenfunctions corresponding to the eigenvalues $\lambda_n=-n^2$, $n\in\mathbb{N}$. Note that $w_n\in\mathrm{L}^p([0,\pi];\mathbb{R}),$ for all $p\in[2,\infty)$ and hence for $y\in\mathbb{X}'$, $\mathrm{T}^*(t)y=\sum\limits_{n=1}^{\infty}e^{-n^2 t}\langle y,w_n\rangle w_n$.

 \subsection{Nonlinear non-autonomous diffusion system}
 Let us now consider the following nonlinear non-autonomous diffusion control system (\cite{FY,MI}):
 \begin{equation}\label{57}
 \left\{
 \begin{aligned}
 \frac{\partial y(t,\xi)}{\partial t}&=a(t,\xi)\frac{\partial^2y(t,\xi)}{\partial \xi^2}+h(t,y(t,\xi))+\eta z(t,\xi), \ \text{ for } \ t\in [0,T], x\in[0,\pi], \\
 y(t,0)&=y(t,\pi)=0, \ \text{ for }\  t\in[0,T], \\
 y(0,\xi)&=\varphi(\xi), \ \text{ for }\ \xi\in[0,\pi]. 
 \end{aligned}
 \right.
 \end{equation}
 We need the following assumptions $a(\cdot,\cdot)$, $\eta$ and $z(\cdot,\cdot)$. 
 \begin{Ass}\label{ass5.1}
 	Let $a(\cdot,\cdot)$, $\eta$ and $z(\cdot,\cdot)$ satisfy the following assumptions: 
 \begin{enumerate}
 	\item [(i)] $a(t,\xi)\geq \delta>0$ and $a(t,x)\in\mathrm{C}^{0,\mu}([0,T];\mathrm{C}([0,\pi]))$, that is, $a(t,x)\in\mathrm{C}([0,\pi];\mathbb{R})$ is uniformly H\"older continuous of order $0<\mu\leq 1$ with respect to the variable $t\in\mathbb{R}$,
 	\item [(ii)] $\eta>0$ and $z:[0,T]\times[0,\pi]\to[0,\pi]$  is continuous in $t$. 
 \end{enumerate}
\end{Ass}
Let us now show that under the Assumption \eqref{ass5.1}-(i), $\mathrm{A}(t)$ generates a \emph{unique evolution system} $\{\mathrm{U}(t,s)\}_{t\geq s}$ on $\mathbb{X}=\mathrm{L}^p([0,\pi];\mathbb{R})$.  For each $t\in[0,T]$, we take $\mathrm{D}(\mathrm{A}(t))=\mathrm{D}(\mathrm{A})=\mathrm{W}^{2,p}([0,\pi];\mathbb{R})\cap\mathrm{W}^{1,p}_0([0,\pi];\mathbb{R})$. Next, we define the operator $\mathrm{A}(t):\mathrm{D}(\mathrm{A}(t))\subset\mathbb{X}\to\mathbb{X}$  as follows:
\begin{align}
\mathrm{A}(t)x(\xi)=a(t,\xi)\mathrm{A}x(\xi)=a(t,\xi)x''(\xi), \ \text{ for }\ x\in\mathrm{D}(\mathrm{A}(t)), \ t\in[0,T],\ \xi\in[0,\pi],
\end{align}
the operator is closed and the domain $\mathrm{D}(\mathrm{A}(t))=\mathrm{D}$ and is dense in $\mathbb{X}$, and hence (P1) is satisfied. From, Assumption \eqref{ass5.1}-(i), it is clear that there exists a constant $M>0$ such that 
\begin{align}\label{5.9}
\sup_{(t,x)\in[0,T]\times[0,\pi]}|a(t,\xi)|\leq M.
\end{align}
Next, we consider the following Sturm-Liouville system:
\begin{equation}\label{59}
\left\{
\begin{aligned}
\left(\lambda\mathrm{I}-\mathrm{A}(t)\right)f(\xi)&=g(\xi), \ 0<\xi<\pi,\\
f(0)=f(\pi)&=0, 
\end{aligned}
\right.
\end{equation}
Since $a(t,x)>\delta>0$, the equation \eqref{59} can be written  as 
\begin{align}\label{511}
\left(\frac{\lambda\mathrm{I}}{a(t,\xi)}-\Delta\right)f(\xi)=\frac{g(\xi)}{a(t,\xi)},
\end{align}
where $\Delta f(\xi)=f''(\xi)$. Multiplying both sides of \eqref{511} by $f|f|^{p-2}$ and then integrating over $[0,\pi]$, we obtain
\begin{align}\label{512}
\lambda\int_0^{\pi}\frac{|f(\xi)|^p}{a(t,\xi)}\mathrm{d}\xi+(p-1)\int_0^{\pi}|f(\xi)|^{p-2}|f'(\xi)|^2\mathrm{d}\xi=\int_0^{\pi}\frac{g(\xi)}{a(t,\xi)}f(\xi)|f(\xi)|^{p-2}\mathrm{d}\xi.
\end{align}
Applying H\"older's inequality and \eqref{5.9}, we further have 
\begin{align}
\frac{\lambda}{M}\int_0^{\pi}|f(\xi)|^p\mathrm{d}\xi\leq \lambda\int_0^{\pi}\frac{|f(\xi)|^p}{a(t,\xi)}\mathrm{d}\xi\leq\frac{1}{\delta} \left(\int_0^{\pi}|g(\xi)|^p\mathrm{d}\xi\right)^{\frac{1}{p}}\left(\int_0^t|f(\xi)|^p\mathrm{d}\xi\right)^{\frac{p-1}{p}},
\end{align}
and hence we have 
\begin{align}
\|\mathrm{R}(\lambda,\mathrm{A}(t))g\|_{\mathrm{L}^p}=\|f\|_{\mathrm{L}^p}\leq\frac{M}{\lambda\delta}\|g\|_{\mathrm{L}^p}, \ \text{for all } \ \lambda>0,
\end{align}
which ensures the condition (P2). Let us now consider
\begin{align}
\|(\mathrm{A}(t)-\mathrm{A}(s))\mathrm{A}^{-1}(\tau)f\|_{\mathrm{L}^p}&=\|(a(t,\xi)-a(s,\xi))a(\tau,\xi)^{-1}f\|_{\mathrm{L}^p}\nonumber\\&\leq \sup_{\xi\in[0,\pi]}|a(t,\xi)-a(s,\xi)|\sup_{\xi\in[0,\pi]}a(\tau,\xi)^{-1}\|f\|_{\mathrm{L}^p}\nonumber\\&\leq\frac{C}{\delta}|t-s|^{\mu}\|f\|_{\mathrm{L}^p},
\end{align}
so that we have $\|(\mathrm{A}(t)-\mathrm{A}(s))\mathrm{A}^{-1}(\tau)\|_{\mathcal{L}(\mathbb{X})}\leq\frac{C}{\delta}|t-s|^{\mu}$ and hence the condition (P3) holds true. The compactness of the resolvent operator $\mathrm{R}(\lambda,\mathrm{A}(t)),$ for condition (P4) can be established in a similar way as in the previous subsection. Thus, using Lemmas \ref{lem2.1} and \ref{lem2.2}, one can assure the existence of a unique evolution system $\left\{\mathrm{U}(t,s):0\leq s\leq t\leq T\right\}$, which is compact for $t-s> 0$. 
Let us now define $$u(t)(\xi):=y(t,\xi)\ \text{ for }\ t\in[0,T]\ \text{ and }\ \xi\in[0,\pi].$$ Then nonlinear function $f:[0,T]\times\mathbb{X}\to\mathbb{X}$ is given by $$f(t,x(t))(\xi)=h(t,y(t,\xi)).$$ Let the operator $\mathrm{B}:\mathrm{L}^2([0,\pi];\mathbb{R})\to\mathbb{X}$ be a bounded linear map defined by $$\mathrm{B}(u(t))(\xi)=u(t)(\xi)=\eta z(t,\xi),\ \text{ for }\ (t,\xi)\in[0,T]\times[0,\pi].$$ With the above notations, \eqref{57} can be written in the abstract form as: 
\begin{equation}\label{517}
\left\{
\begin{aligned}
\frac{\mathrm{d}x(t)}{\mathrm{d}t}&=\mathrm{A}(t)x(t)+f(t,x(t))+\mathrm{B}(u(t)),\\
x(0)&=\varphi.
\end{aligned}\right.
\end{equation}
The boundary conditions has been taken care by the definition of domain of the operator $\mathrm{A}(t)$ and into the requirement that $x(t)\in\mathrm{ D}(\mathrm{A}),$ for all $t\geq 0$. For the corresponding linear system to \eqref{517} to be approximately controllable, we know from Theorem \ref{thm3.2} that $\mathrm{U}^*(T,t)\mathrm{B}^*x'=0$ on  $[0,T]$ implies $x'=0$ (see Remark \ref{rem3.2} also). For $x'\in\mathbb{X}'$, we consider 
\begin{align*}
\mathrm{B}^*\mathrm{U}^*(T,t)x'=0\Rightarrow \mathrm{U}(T,t)x'=0\Rightarrow x'=0,
\end{align*}
and hence the linear system corresponding to \eqref{517} is approximately controllable. Thus, $\Lambda_T$ is positive; that is, $\langle x', \Lambda_T x' \rangle > 0,$ for all nonzero $x'\in\mathbb{X}'$ and hence from Theorem 2.3, \cite{M}, we obtain  that for every $x\in\mathbb{X}$, $\lambda\mathrm{R}(\lambda,\Lambda_T)(x)$ converges to zero as $\lambda\downarrow 0$ in strong operator topology. Thus, the Assumption \ref{ass2.2}-(A3) holds true. If we let, $f(t,x(t))=\sin(x(t))$, then Assumptions (A1) and (A2) are satisfied with constants $K=1$ and $N=1$, respectively. Invoking Theorem \ref{thm4.2}, we finally obtain that the nonlinear system \eqref{517} (equivalently the system \eqref{57}) is approximately controllable. 

\begin{rem}
	If $a(t,\xi)$ in \eqref{57} is independent of $\xi$, then the evolution system  $\{\mathrm{U}(t,s)\}_{t\geq s}$ can be explicitly written as 
$
	\mathrm{U}(t,s)x=\mathrm{T}\left(\int_s^{t}a(\tau)d\tau\right)x,
$
	for $x\in\mathbb{X}$. 
	\end{rem}

\end{document}